\documentclass[10pt,leqno]{article}
\usepackage{amsmath,amssymb,amsthm,latexsym}
\usepackage{indentfirst}
\usepackage{amsmath}
\newtheorem{Definition}{\bf \large Definition}[section]
\newtheorem{Theorem}{\bf \large Theorem}[section]
\newtheorem{PROPOSITION}{\bf \large Proposition}[section]
\newtheorem{Corollary}{\bf \large Corollary}[section]
\newtheorem{lemma}{\bf \large Lemma}[section]

\newtheorem{Remark}{\bf \large Remark}[section]
\newtheorem{Lemma}{\bf \large Lemma}[section]
\newtheorem{Conjecture}{\bf \large Conjecture}[section]

\title{\textbf{M\"{o}bius and Laguerre geometry of Dupin Hypersurfaces}}
\author {{ Tongzhu Li$^1$, Jie Qing$^2$}, Changping Wang$^3$ \\
\small{1 Department of Mathematics, Beijing Institute of
Technology,} \\
\small{Beijing,100081,China, E-mail:litz@bit.edu.cn.} \\
\small{2 Department of Mathematics, University of California, Santa Cruz,}\\
\small{CA, 95064, U.S.A., E-mail:qing@ucsc.edu.}\\
\small{3 College of Mathematics and Computer Science, Fujian Normal University,}\\
\small{Fuzhou, 350108, China, E-mail:cpwang@fjnu.edu.cn}}
\date{}

\begin{document}
\maketitle
\begin{abstract}
In this paper we show that a Dupin hypersurface with constant M\"{o}bius curvatures is M\"{o}bius equivalent to either an isoparametric
hypersurface in the sphere or a cone over an isoparametric hypersurface in a sphere. We also show that a Dupin hypersurface with constant
Laguerre curvatures is Laguerre equivalent to a flat Laguerre isoparametric hypersurface.
These results solve the major issues related to the conjectures of Cecil et al on the classification of Dupin hypersurfaces.
\end{abstract}

\vskip 1 cm
\section{Introduction}
Let $M^n$ be an immersed hypersurface in Euclidean space $\mathbb{R}^{n+1}$. A curvature surface of $M^n$ is a smooth connected submanifold $S$
such that for each point $p\in S$, the tangent space $T_pS$ is equal to a principal space of the shape operator $\mathcal{A}$ of $M^n$ at $p$.
The hypersurface $M^n$ is called Dupin hypersurface if, along each curvature surface, the associated principal curvature is constant. The Dupin
hypersurface $M^n$ is called proper Dupin if the number $r$ of distinct principal curvatures is constant on $M^n$.
Both Dupin and properness are invariant under the group of Lie sphere transformations, which was verified by Pinkall \cite{pink2}. The group of
 Lie sphere transformations is generated by its two subgroups: the group of
M\"{o}bius transformations and the group of Laguerre transformations. Hence, due to M\"{o}bius invariance, the theory of Dupin
submanifolds is essentially the same whether it is considered in $\mathbb{R}^{n+1}$,  $\mathbb{S}^{n+1}$, or $\mathbb{H}^{n+1}$.

Dupin surfaces were first studied by Dupin \cite{Dupin} in 1822. Since then, Dupin hypersurfaces have been studied
extensively (cf. \cite{cecil3, cecil4, cecil6,cecil7,cecil8,cecil9,gro,miya1,miya2,miya3,pink1,riv3,st}).
The classification of Dupin hypersurfaces is far from complete, especially for higher
dimensions. An important class of examples are the isoparametric hypersurfaces in
$\mathbb{R}^{n+1}$, $\mathbb{H}^{n+1}$, and $\mathbb{S}^{n+1}$. An isoparametric hypersurface is a hypersurface with constant principal curvatures.
In $\mathbb{R}^{n+1}$ as well as $\mathbb{H}^{n+1}$, an isoparametric hypersurface has no more than 2 distinct principal curvatures and is completely
classified (cf. \cite{cecil3}).
On the other hand, in $\mathbb{S}^{n+1}$, there are many more examples (cf. \cite{cecil5,Dorf,im,te1,tho2}). M\"{u}nzner (\cite{mu1,mu2})
showed that the number $r$ of distinct principal curvatures of an isoparametric hypersurface in $\mathbb{S}^{n+1}$ must be $1, 2, 3, 4$ or $6$.
Cartan \cite{car0} classified those with $r\leq 3$.

Thorbergsson \cite{tho1} showed the restriction that $r=1,2,3,4$ or $6$  on the number of distinct principal curvatures also
holds for compact proper Dupin hypersurfaces embedded in $\mathbb{S}^{n+1}$. Compact proper Dupin hypersurfaces with $r\leq 3$
are completely classified (cf. \cite{cecil10,miya1}). However, it is a different story if the compactness is dropped. In fact, Pinkall \cite{pink2} discovered the basic
constructions of building tubes, cylinders, cones and surfaces of revolution over a Dupin
hypersurface. It is important to note that these constructions may yield a compact proper
Dupin hypersurface only if the original one is a sphere (cf. \cite{cecil3,pink2}).
A Dupin hypersurface which is locally equivalent by a Lie sphere transformation to a
hypersurface $M^n$ obtained by one of these four basic constructions is said to be reducible, otherwise,
the Dupin hypersurface is said to be irreducible. Local classifications have been obtained for irreducible connected proper Dupin
hypersurfaces with $r\leq 3$ (cf. \cite{cecil7,cecil8,Nie1,pink2}).

In all of the above cases when $r\leq 3$,  compact (or irreducible) Dupin hypersurfaces are known to be Lie equivalent to isoparametric hypersurfaces (cf. \cite{cecil8,miya1,Nie1}). In addition,
Stolz \cite{st} in the cases $r=4$ and Grove and Halperin \cite{gro} in the cases $r=6$ have shown that the multiplicities of the principal curvatures of a
compact proper Dupin hypersurface must be the same as that of an isoparametric hypersurface. Hence it was conjectured that,
at least for compact cases, proper Dupin hypersurface is always  Lie equivalent to an isoparametric hypersurface
(see for instance \cite[p.184]{cecil7}). However, this conjecture was shown to be false by Pinkall and Thorbergsson \cite{pink3} for $r=4$ and separately by Miyaoka
and Ozawa \cite{miya4} for $r=4$ and $r=6$.  The compact proper Dupin hypersurfaces for counterexamples in \cite{pink3,miya4} have non-constant Lie curvatures.
For an oriented hypersurface $M^n$ with $r (\geq 4)$ distinct principal curvatures $\lambda_1,\cdots,\lambda_r$, Miyaoka \cite{miya2} introduced
Lie curvatures as the cross-ratios of the principal curvatures
$$\Psi_{ijst}=\frac{(\lambda_i-\lambda_j)(\lambda_t-\lambda_s)}{(\lambda_i-\lambda_s)(\lambda_t-\lambda_j)}$$
and verified that Lie curvatures $\Psi_{ijst}$ are invariant under Lie sphere transformations.
Obviously, that the Lie curvatures are constant is a necessary condition for a Dupin hypersurface to
be Lie equivalent to an isoparametric hypersurface.  Therefore Cecil, Chi and Jensen \cite{cecil7} proposed the following conjecture.

\begin{Conjecture}\label{c1} (\cite{cecil7}) Every compact connected proper Dupin hypersurface with four or six principal curvatures and constant Lie curvatures is Lie equivalent
to an isoparametric hypersurface in a sphere.
\end{Conjecture}

Analogously, on local classifications of irreducible connected proper Dupin hypersurfaces, Cecil, Chi and Jensen \cite{cecil7}
proposed the following conjecture.

\begin{Conjecture}\label{c2} (\cite{cecil7}) If $M^n$ is an irreducible connected proper Dupin hypersurface with four  principal curvatures having respective multiplicities $m_1,m_2,m_3,m_4$ and constant Lie curvature, then $m_1=m_2,~m_3=m_4$, and $M^n$ is Lie equivalent to an isoparametric hypersurface in a sphere.
\end{Conjecture}

In \cite{cecil6}, Cecil, Chi and Jensen have verified both Conjecture \ref{c1} and Conjecture \ref{c2} for Dupin hypersurfaces with four principal curvatures and
multiplicities
$m_1=m_2\geq 1,~m_3=m_4=1$. In this paper we will shed different lights on Conjecture \ref{c1} and Conjecture \ref{c2}.
More precisely, we will consider M\"{o}bius (Laguerre) curvatures instead of Lie curvatures and show Conjecture \ref{c1} and Conjecture \ref{c2} hold
in stronger ways where Lie equivalence is replaced by M\"{o}bius (Laguerre) equivalence respectively.

For an oriented hypersurface $M^n$ with $r (\geq 3)$ distinct principal curvatures $\lambda_1,\cdots,\lambda_r$, the M\"{o}bius curvatures are defined by
$$\mathbb{M}_{ijs}=\frac{\lambda_i-\lambda_j}{\lambda_i-\lambda_s}.$$
It is known that the M\"{o}bius curvatures $\mathbb{M}_{ijs}$ are invariant under the M\"{o}bius transformations but not  under Lie sphere transformations in general
(cf. \cite{miya2}). Obviously, Lie curvatures are products of two M\"{o}bius curvatures
$$\Psi_{ijst}=\mathbb{M}_{ijs} \mathbb{M}_{tsj}$$
and therefore M\"{o}bius curvatures are finer than Lie curvatures.  On a seemingly different thread,  Hu, Li, and Wang \cite{hu1,hu2,hu3} classified the hypersurfaces with vanishing M\"{o}bius form and constant M\"{o}bius principal curvatures, which are called M\"{o}bius isoparametric hypersurfaces, provided that the dimension of the hypersurface or the
number of distinct principal curvatures is small.
Until Rodrigues and Tenenblat \cite{rod} observed that an oriented hypersurface is a Dupin hypersurface with constant M\"{o}bius curvatures if and only if it is a
M\"{o}bius isoparametric hypersurface.

Our first main result is the following classification theorem:

\begin{Theorem}\label{th1}
Let $M^n$ be a Dupin hypersurface in $\mathbb{R}^{n+1}$ with $r (\geq 3)$ distinct principal curvatures. If the  M\"{o}bius curvatures are constant, then locally $M^n$
is M\"{o}bius equivalent to one of the following hypersurfaces: \\
(1) the image of the stereograph projection of an isoparametric hypersurface in $\mathbb{S}^{n+1}$;\\
(2) a cone over an isoparametric hypersurface in $\mathbb{S}^k\subset \mathbb{R}^{k+1}\subset\mathbb{R}^{n+1}$.
\end{Theorem}

A Dupin hypersurface with constant M\"{o}bius curvatures turns out to be proper
(cf. Corollary \ref{corm}). As a consequence of the constraint on the number of distinct principal curvatures for isoparametric hypersurfaces established by M\"{u}nzner (\cite{mu1,mu2}),
we may conclude

\begin{Corollary}\label{cor1}
Let $M^n$ be a Dupin hypersurface in $\mathbb{R}^{n+1}$ with $r (\geq 3)$ distinct principal curvatures. If the M\"{o}bius curvatures are constant, then
$r=3, 4, 5, 6, 7.$
\end{Corollary}

As argued in \cite{cecil6} based on the analyticity of Dupin hypersurfaces established in \cite{cecil6.5}, we can use Theorem \ref{th1} to solve some major issues
related  Conjecture \ref{c1} and Conjecture \ref{c2} on the classification of Dupin hypersurfaces.

\begin{Corollary}\label{cor2}
Let $M^n$ be a compact connected Dupin hypersurface with $r (\geq 3)$ distinct principal curvatures. Then
$M^n$ is M\"{o}bius equivalent to an isoparametric hypersurface in $\mathbb{S}^{n+1}$ if and only if its M\"{o}bius curvatures are all constant.
\end{Corollary}
and
\begin{Corollary}\label{cor3}
Let $M^n$ be an irreducible connected Dupin hypersurface with $r (\geq 3)$ distinct principal curvatures. Then
$M^n$ is M\"{o}bius equivalent to an isoparametric hypersurface in $\mathbb{S}^{n+1}$ if and only if its M\"{o}bius curvatures are all constant.
\end{Corollary}

For an oriented hypersurface $M^n$ in $\mathbb{R}^{n+1}$ with non-vanishing principal curvatures $\lambda_1,\lambda_2,\cdots,\lambda_r$,
let $R_i=\frac{1}{\lambda_i}$ be the curvature radius. Then one can define the Laguerre curvatures of $M^n$ as
$$\Upsilon_{ijs}=\frac{R_i-R_j}{R_i-R_s}.$$
It is clear that $\Upsilon_{ijk}$ are invariant under Laguerre transformations in the light of \eqref{lag-cur-inv} in Section \ref{lag-inv}.
Again, obviously, the Lie curvature is a product of  two Laguerre curvatures
$$\Psi_{ijst}=\Upsilon_{ijs}\Upsilon_{tsj}$$
and therefore Laguerre curvatures are finer than Lie curvatures. Analogous to M\"{o}bius cases,
it turns out, as expected, an oriented hypersurface is a Dupin hypersurface with constant Laguerre curvatures if and only if its Laguerre form vanishes and its Laguerre principal
curvatures are all constant (see Proposition \ref{pro1} and also \cite{mtenen}).

Our second main theorem is the following result:

\begin{Theorem}\label{th2}
Let $M^n$ be a  Dupin hypersurface in $\mathbb{R}^{n+1}$ with $r (\geq 3)$ distinct non-vanishing principal curvatures. Then $M^n$
is Laguerre equivalent to flat Laguerre  isoparametric hypersurface in $ \mathbb{R}^{n+1}$ if and only if the Laguerre curvatures are all constant.
\end{Theorem}

Flat Laguerre  isoparametric hypersurfaces will be reviewed in Section \ref{lag-example}.
We note that a flat Laguerre isoparametric hypersurface is a Dupin hypersurface with any given number $k$ of principal curvatures with any prescribed multiplicities $m_1,\cdots,m_k$ and that a flat Laguerre isoparametric hypersurface is reducible and non-compact.

Our approach is to recognize M\"{o}bius (Laguerre) second fundamental form $B$ ($\mathbb{B}$) and Blaschke (Laguerre) tensor $A$ ($\mathbb{L}$)
are commuting isoparametric tensors. Isoparametric tensors on a Riemannian manifold $(M^n,g)$ are the Codazzi tensors with constant
eigenvalues. The name comes from the fact that the second fundamental form of an isoparametric hypersurface in space forms is an isoparametric tensor.
We will play with the integrability conditions \eqref{equa4} and \eqref{2.8} and generalized Cartan identity \eqref{tensor4} to
pin down the specific behaviors of M\"{o}bius (Laguerre) second fundamental form $B$ ($\mathbb{B}$) and Blaschke (Laguerre) tensor $A$ ($\mathbb{L}$).
In M\"{o}bius cases it turns out miraculously we are able to show that either the M\"{o}bius second fundamental form and Blaschke tensor are linearly dependent
(cf. Definition \ref{lin-iso}) or they behave as in \eqref{cone-form} and \eqref{cone-inv}. Then one may conclude the Dupin hypersurface is M\"{o}bius equivalent to
an isoparametric hypersurface in the sphere in the former cases following \cite[Main Theorem]{li2}, and to a cone over an isoparametric hypersurface in a sphere
in the latter cases in the light of Theorem \ref{cone-type}. In Laguerre cases the situation is much simpler. We will be able to show that the Laguerre second fundamental
form $\mathbb{B}$ in our consideration is in fact parallel. Then Theorem \ref{th2} follows from the classification result in \cite{lit1}.

We now give a brief outline of the paper. In Sections \ref{mobius-inv}, we will recall some facts about the M\"{o}bius geometry of a hypersurface in $\mathbb{R}^{n+1}$
and \cite[Main Theorem]{li2}.
In Section \ref{cone-char}, we will give a M\"{o}bius characterization of the cone hypersurfaces.
In Section \ref{proof-th1}, we will present the proof of our first main classification Theorem \ref{th1}.
In Section \ref{lag-inv}, we will recall some facts about the Laguerre geometry of a hypersurface in $\mathbb{R}^{n+1}$.
In Section \ref{proof-th2}, we will first review two families of  examples of Dupin hypersurface and then we will present
the proof of our second main Theorem \ref{th2}. In Appendix A, we will discuss some properties of isoparametric tensors and applications.


\section{M\"{o}bius invariants of hypersurfaces in $\mathbb{R}^{n+1}$}\label{mobius-inv}

In this section, to set the notations, we will briefly review the M\"obius geometry of hypersurfaces
in $\mathbb{R}^{n+1}$ via the Minkowski spacetime $\mathbb{R}^{n+3}_1$.
For details readers are referred to $\cite{ak,liu,w}$.
We observe there is a straightforward way to see that a Dupin hypersurface of constant M\"{o}bius curvatures
is always proper. We will also recall the characterization of Dupin hypersurfaces of constant M\"{o}bius curvatures in terms of
M\"{o}bius invariants given in \cite{rod}. We will also derive the characterization of Dupin hypersurfaces that are M\"{o}bius
equivalent to isoparametric hypersurfaces based on \cite[Main Theorem]{li2}.

Let $\mathbb{R}^{n+3}_1$ be the Minkowski spacetime, i.e., $\mathbb{R}^{n+3}$ with the
standard spacetime metric
\[\langle x,y\rangle=-x_0y_0+x_1y_1+\cdots+x_{n+2}y_{n+2}\] for
$x=(x_0,x_1,\cdots,x_{n+2})$ and $y=(y_0,y_1,\cdots,y_{n+2})$.
One may identify the conformal round sphere $\mathbb{S}^{n+1}$ as the projective positive light cone
$$C^{n+2}_+=\{y=(y_0,y_1)\in R\times \mathbb{R}^{n+2}|\langle y,y\rangle=0,y_0>0\}\subset \mathbb{R}^{n+3}_1.$$
Let $\textup{O}^+(n+2,1)$ be the Lorentz group of linear transformations of $\mathbb{R}^{n+3}_1$ that preserve
the time orientation and the spacetime metric, and let
$\textup{M}(\mathbb{S}^{n+1})$ be the group of M\"{o}bius transformations of $\mathbb{S}^{n+1}$. One knows from Liouville Theorem that
the group $\textup{M}(\mathbb{R}^{n+1})$ of M\"{o}bius transformations on $\mathbb{R}^{n+1}$ is the same as $\textup{M}(\mathbb{S}^{n+1})$.
It is then useful to mention the natural isomorphism
$$
L: \textup{O}^+(n+2, 1)\to \textup{M}(\mathbb{S}^{n+1})= \textup{M}(\mathbb{R}^{n+1}).
$$

Let $f:M^{n}\rightarrow \mathbb{R}^{n+1} $ be a hypersurface without umbilical points and
$\{e_i\}$ be an orthonormal basis with respect to the
induced metric $I=df\cdot df$ with the dual basis $\{\theta_i\}$.
Let $II=\sum_{ij}h_{ij}\theta_i\theta_j$ and
$H=\sum_i\frac{h_{ii}}{n}$ be the second fundamental form and the
mean curvature of $f$ respectively. To study the M\"{o}bius geometry of $f$, as in \cite{liu,w},
one considers the M\"{o}bius position vector
$$Y=\rho(f)\left(\frac{1+|f|^2}{2},\frac{1-|f|^2}{2},f\right): M^n\rightarrow C^{n+2}_+\subset \mathbb{R}^{n+3}_1
$$
and the M\"{o}bius metric
$$
g=<dY,dY>=(\rho(f))^2 df\cdot df,
$$
where $(\rho (f))^2=\frac{n}{n-1}(|II|^2-nH^2)$. One basic fact for this approach is

\begin{Lemma} Suppose that $f: M^n\to \mathbb{R}^{n+1}$ is an immersed hypersurface and
$$
Y=\rho(f)\left(\frac{1+|f|^2}{2},\frac{1-|f|^2}{2},f\right): M^n\rightarrow C^{n+2}_+\subset \mathbb{R}^{n+3}_1
$$
is the M\"{o}bius position vector of $f$. Then, for any $T\in \textup{O}^+(n+2, 1)$, we have
\begin{equation}\label{co-var}
TY = \rho(L(T)f) \left(\frac{1+|L(T)f|^2}{2},\frac{1-|L(T)f|^2}{2},L(T)f\right): M^n\rightarrow C^{n+2}_+\subset \mathbb{R}^{n+3}_1
\end{equation}
and therefore the M\"{o}bius metric $g$ stays invariant.
\end{Lemma}

To build a moving frame along $Y$ in $\mathbb{R}^{n+3}_1$,
as in \cite{liu,w}, one starts with the so-called conformal Gauss map
$$
\xi=H\left(\frac{1+|f|^2}{2},\frac{1-|f|^2}{2},f\right)+\left(f\cdot e_{n+1},-f\cdot e_{n+1},e_{n+1}\right)
$$
that represents the mean curvature sphere, where  $e_{n+1}$ is the unit normal vector field of $f$ in
$R^{n+1}$. One can pick up a moving frame $\{Y_1=Y_*(E_1),\cdots,Y_n=Y_*(E_n)\}$ for the tangent space of $Y$ along $M^n$ with its dual
$\{\omega_1,\cdots,\omega_n\}$. To complete a moving frame, as in \cite{liu,w}, one chooses another null normal
vector $N$ to $Y$ in $\mathbb{R}^{n+3}_1$ such that $<N,Y>= 1$. Thus $\{Y,N,Y_1,\cdots,Y_n,\xi\}$ forms a moving frame in
$\mathbb{R}^{n+3}_1$ along $Y$ and the structure equations are:
\begin{equation}\label{stru}
\begin{split}
&dY=\sum_iY_i\omega_i,\\
&dN=\sum_{ij}A_{ij}\omega_iY_j+\sum_iC_i\omega_i\xi,\\
&dY_i=-\sum_jA_{ij}\omega_jY-\omega_iN+\sum_j\omega_{ij}Y_j+\sum_jB_{ij}\omega_j\xi,\\
&d\xi=-\sum_iC_i\omega_iY-\sum_{ij}\omega_iB_{ij}Y_j,
\end{split}
\end{equation}
where $\omega_{ij}$ is the connection form of the M\"{o}bius metric
$g$ with respect to the dual $\{\omega_1,\cdots,\omega_n\}$ and the range of Latin indices are in $1, 2, \cdots, n$. The tensors
$$
A=\sum_{ij}A_{ij}\omega_i\otimes\omega_j,~~
B=\sum_{ij}B_{ij}\omega_i\otimes\omega_j,~~
C=\sum_iC_i\omega_i$$ are called the
Blaschke tensor, the M\"{o}bius second fundamental form and the M\"{o}bius form of $f$ respectively (cf. \cite{liu,w}).
In \cite{w},  the integrability conditions for $\{A, B, C\}$ are identified as
\begin{eqnarray}
&&A_{ij,k}-A_{ik,j}=B_{ik}C_j-B_{ij}C_k,\label{equa1}\\
&&C_{i,j}-C_{j,i}=\sum_k(B_{ik}A_{kj}-B_{jk}A_{ki}),\label{equa2}\\
&&B_{ij,k}-B_{ik,j}=\delta_{ij}C_k-\delta_{ik}C_j,\label{equa3}\\
&&R_{ijkl}=B_{ik}B_{jl}-B_{il}B_{jk}
+\delta_{ik}A_{jl}+\delta_{jl}A_{ik}
-\delta_{il}A_{jk}-\delta_{jk}A_{il},\label{equa4}\\
&&R_{ij}:=\sum_kR_{ikjk}=-\sum_kB_{ik}B_{kj}+(tr{\bf
A})\delta_{ij}+(n-2)A_{ij},\label{equa5}\\
&&\sum_iB_{ii}=0, ~~\sum_{ij}(B_{ij})^2=\frac{n-1}{n},~~ tr{\bf
A}=\frac{1}{2n}(1+n^2\kappa),\label{equa6}
\end{eqnarray}
where $R_{ijkl}$ denote the curvature tensor of $g$,
$\kappa=\frac{1}{n(n-1)}\sum_{ij}R_{ijij}$ is its normalized
scalar curvature.  Most importantly, it was shown in \cite{w} that  $\{g, {B}\}$
determines the hypersurface $f$ up to M\"{o}bius
transformations provided that $n\geq 3$.

We would also like to recall from \cite{liu,w} how $\{A, B, C\}$ can be calculated in terms of the geometry of $f$ in $\mathbb{R}^{n+1}$:
\begin{equation}\label{re2}
\begin{split}
B_{ij}&=\rho^{-1}(h_{ij}-H\delta_{ij}),\\
C_i&=-\rho^{-2}[e_i(H)+\sum_j(h_{ij}-H\delta_{ij})e_j(\log\rho)],\\
A_{ij}&=-\rho^{-2}[Hess_{ij}(\log\rho)-e_i(\log\rho)e_j(\log\rho)-Hh_{ij}]\\
&-\frac{1}{2}\rho^{-2}(|\nabla \log\rho|^2+H^2)\delta_{ij},
\end{split}
\end{equation}
where Hessian and $\nabla$ are with respect to $I=df\cdot df$.
The eigenvalues of $B$ are called the M\"{o}bius principal curvatures of $f$.
Let $\{b_1,\cdots,b_n\}$ be the M\"{o}bius principal curvatures and
$\{\lambda_1,\cdots,\lambda_n\}$ be the principal curvatures of $f$,
then, from \eqref{re2},
\begin{equation}\label{prin-cur}
b_i=\rho^{-1}(\lambda_i-H).
\end{equation}
Clearly the number of distinct M\"{o}bius principal curvatures is the same as
that of principal curvatures of $f$ and
\begin{equation}\label{moecur}
\mathbb{M}_{ijk}=\frac{\lambda_i-\lambda_j}{\lambda_i-\lambda_k}=\frac{b_i-b_j}{b_i-b_k},
\end{equation}
which confirms that the M\"{o}bius curvatures are M\"{o}bius invariants. It is then rather easily seen from \eqref{equa6} that, if
M\"{o}bius curavtures $\mathbb{M}_{ijk}$ are constant for all $1\leq i, j, k \leq n$, then all M\"{o}bius principal curvatures $\{b_i\}$ are constant.
\begin{PROPOSITION}\label{pro4}
Let $f:M^{n}\rightarrow \mathbb{R}^{n+1}$ be an immersed  hypersurface with $r (\geq 3)$ distinct principal curvatures. Then the M\"{o}bius curvatures
$\mathbb{M}_{ijk}$ are constant if and only if the M\"{o}bius principal curvatures $\{b_1,\cdots,b_n\}$ are constant.
\end{PROPOSITION}
\begin{proof} It suffices to prove that the M\"{o}bius curvatures $\mathbb{M}_{ijk}$ are constant implies all M\"{o}bius principal curvatures $b_i$ are constant.
First, for any tangent vector $X\in TM^n$, it is not hard to calculate that
$$\frac{X(b_i)-X(b_j)}{b_i-b_j}=\frac{X(b_i)-X(b_k)}{b_i-b_k}= \frac{X(b_j)-X(b_k)}{b_j-b_k} $$
from $\mathbb{M}_{ijk}$ being constant for all $1\leq i,j,k\leq n$.
Hence there exist $\mu$ and  $d$ such that
\begin{equation}\label{lb11}
X(b_j)=\mu b_j+d~~\text{for } j=1,\cdots,n.
\end{equation}
It is then immediate that \eqref{equa6} implies $d=0$ and $b_1X(b_1)+\cdots+b_nX(b_n)=0$, which implies $\mu=0$. Thus all $b_1,\cdots, b_n$ are constant.
\end{proof}
As a consequence of Proposition \ref{pro4} and \eqref{prin-cur}, one easily derives

\begin{Corollary}\label{corm}
A Dupin hypersurface of constant M\"{o}bius curvature is always proper.
\end{Corollary}

As another consequence of Proposition \ref{pro4}, we can characterize the Dupin hypersurfaces of constant M\"{o}bius curvature in terms of M\"{o}bius invariants,
which were observed by Rodrigues L.A. and Tenenblat K. \cite{rod}.
In fact, from \eqref{re2}, we have
$$C_i=-\rho^{-2}[e_i(H)+\sum_je_j(\rho B_{ij})-\rho e_j(B_{ij})]=-\rho^{-2}[e_i(\lambda_i)-\rho e_i(b_i)].$$
Hence one easily derives
\begin{Theorem} \label{is1} (\cite{rod}) Let $f:M^{n}\rightarrow \mathbb{R}^{n+1}$ be an immersed hypersurface with $r (\geq 3)$ distinct principal curvatures.
Then it is a Dupin hypersurface of constant M\"{o}bius curvatures if and only if its M\"{o}bius form vanishes and its M\"{o}bius principal curvatures
are all constant, i.e. it is a M\"{o}bius isoparametric hypersurface.
\end{Theorem}

Recall from \cite{hu1, hu2,lih2}, an immersed hypersurface is said to be a M\"{o}bius
isoparametric hypersurface if its M\"{o}bius form vanishes and its M\"{o}bius principal curvatures are all constant. We would like to remind readers that a M\"{o}bius
isoparametric hypersurface is not necessarily M\"{o}bius equivalent to an isoparametric hypersurface.

The following result \cite[Main Theorem]{li2} enables us to characterize Dupin hypersurfaces that are M\"{o}bius equivalent to isoparametric hypersurfaces.

\begin{Theorem}\cite[Main Theorem]{li2} \label{liw}
Let $x:M^n\to \mathbb{S}^{n+1}$ be an immersed hypersurface. Suppose that it satisfies
that $C=0$ and $A=\lambda B+\mu g$ for some functions $\lambda$ and $\mu$.
Then $x$ is M\"{o}bius equivalent to a hypersurface with constant mean curvature and scalar curvature in Euclidean space $\mathbb{R}^{n+1}$, or sphere
$\mathbb{S}^{n+1}$, or hyperbolic space $\mathbb{H}^{n+1}$.
\end{Theorem}

Consequently we have

\begin{Corollary}\label{iso-type}
 Suppose that $f: M^n\to \mathbb{R}^{n+1}$ is a Dupin hypersurface with $r (\geq3)$ distinct principal curvatures and constant M\"{o}bius curvatures.
 Then it is M\"{o}bius equivalent to an isoparametric hypersurface in the sphere
$\mathbb{S}^{n+1}$ if and only if $A =  \lambda B  + \mu g$ for some numbers $\lambda$ and $\mu$.
\end{Corollary}

\begin{proof}
If $A =  \lambda B  + \mu g$ for some numbers $\lambda$ and $\mu$, by the definition of the function $\rho$, it is immediate that $\rho$ is constant when the mean curvature and the scalar curvature of a hypersurface in a space form
are constant.
Then one can conclude that the Dupin hypersurface under the assumptions is M\"{o}bius equivalent to an isoparametric hypersurface in space forms in the light of the first
equation in \eqref{re2}. Therefore, with the assumption that $f$ has $r (\geq 3)$ distinct principal curvatures,  the proof is complete due to known classifications of isoparametric hypersurafces in $\mathbb{R}^{n+1}$ and $\mathbb{H}^{n+1}$ (cf. \cite[Chap. 3]{cecil3}).

On the contrary, If $f$ is  an isoparametric hypersurface in the sphere
$\mathbb{S}^{n+1}$, then the mean curvature and the scalar curvature are constant. Hence $\rho$ is constant and $A =  \lambda B  + \mu g$ for some numbers $\lambda$ and $\mu$ by the equations \eqref{re2}.
\end{proof}


\section{Cones over isoparametric hypersurfaces}\label{cone-char}

Remarkably in \cite{pink2}, Pinkall discovered the cone over a Dupin hypersurface is still a Dupin hypersurface. In fact it is easily seen
that the cone over an isoparametric hypersurface in a sphere is always a Dupin hypersurface of
constant M\"{o}bius curvatures. In this section we will calculate the M\"{o}bius invariants,  and using M\"{o}bius invariants to characterize cones
over isoparametric hypersurfaces in spheres. Let us start with the construction of cones over hypersurfaces in spheres.

\begin{Definition} For $1 \leq k \leq n-1$,
let $u: {M}^k\longrightarrow \mathbb{S}^{k+1}\subset \mathbb{R}^{k+2}$ be an immersed hypersurface in
$\mathbb{S}^{k+1}$. The cone over $u$ in $\mathbb{R}^{n+1}$ is given as
\begin{equation*}
f(t,y,p)=(y,tu(p)): \mathbb{R}^+\times \mathbb{R}^{n-k-1}\times {M}^k\longrightarrow \mathbb{R}^{n+1}.
\end{equation*}
\end{Definition}

It is easily calculated that the first fundamental form of the cone $f$ is $I_f =  dt^2 + |dy|^2 +t^2 I_u$
and the second fundamental form of the cone is $II_f = t~II_u$, where $I_u$ and $II_u$ are the first and
second fundamental forms of the hypersurface $u$ in the sphere $\mathbb{S}^{k+1}$
respectively. The principal curvatures of the cone $f$ are
\begin{equation}\label{princ}
\underbrace{0,\cdots,0}_{n-k},\frac{1}{t}\lambda_1,\cdots,\frac{1}{t}\lambda_k,
\end{equation}
where $\{\lambda_1,\cdots,\lambda_k\}$ are the principal curvatures of $u$. Hence
$$
\rho^2 = \frac  n{n-1}(|II_u|^2 - \frac {k^2}nH_u^2) \frac 1{t^2},
$$
the M\"{o}bius metric of the cone is
\begin{equation}\label{metric}
g = \rho^2 I_f =  \rho_0^2 (\frac {dt^2 + |dy|^2}{t^2} + I_u),
\end{equation}
and the M\"{o}bius position vector of the cone  is
$$
Y(t, y, p) =\frac {\rho_0}t (\frac{1+t^2+|y|^2}{2},\frac{1-t^2-|y|^2}{2}, y, tu(p)):
\mathbb{R}^+\times \mathbb{R}^{n-k-1}\times{M}^k\to
C^{n+2}_+\subset \mathbb{R}^{n+3}_1,
$$
where $\rho_0^2=\frac{n}{n-1}(|II_u|^2-\frac{k^2}{n}H_u^2)$. Note that
\begin{equation}\label{id}
i(t, y) = (\frac{1+t^2+|y|^2}{2t},\frac{1-t^2-|y|^2}{2t},\frac{y}{t}):
\mathbb{R}^+\times \mathbb{R}^{n-k-1}=\mathbb{H}^{n-k}\to \mathbb{H}^{n-k}\subset \mathbb{R}^{n-k+1}_1
\end{equation}
is nothing but the identity map of $\mathbb{H}^{n-k}$, since $\mathbb{R}^+\times \mathbb{R}^{n-k-1}=\mathbb{H}^{n-k}$ is the upper
half-space endowed with the standard hyperbolic metric. We may now rewrite the M\"{o}bius position vector of the cone $f$ as
\begin{equation}\label{mop2}
Y=\rho_0(i(t, y), u):\mathbb{R}^+\times \mathbb{R}^{n-k-1}\times{M}^k \to C^{n+2}_+\subset \mathbb{R}^{n+3}_1.
\end{equation}
Consequently we have

\begin{Lemma}\label{redu}
Let $u:{M}^k\to \mathbb{S}^{k+1}$ be an immersed hypersurface in $\mathbb{S}^{k+1}\subset\mathbb{R}^{k+2}$ and
 \begin{equation*}
\frac 1{\rho_0}Y= (i(t, y), u):\mathbb{R}^+\times \mathbb{R}^{n-k-1}\times{M}^k\to \mathbb{H}^{n-k}\times\mathbb{S}^{k+1}\subset \mathbb{R}^{n+3}_1
\end{equation*}
for smooth positive function $\rho_0$. Suppose that $Y$ is the M\"{o}bius position vector for an immersed hypersurafce
$$f:\mathbb{R}^+\times \mathbb{R}^{n-k-1}\times{M}^k\to \mathbb{R}^{n+1}.$$
Then $f$ is a cone over $u$ and $\rho_0^2=\frac{n}{n-1}(|II_u|^2-\frac{k^2}{n}H_u^2)$.
\end{Lemma}

Lemma \ref{redu} is useful when we derive the cone structure of a Dupin hypersurface $f$ from the cone structure of its M\"{o}bius position vector $Y$. It is also easily seen that

\begin{Lemma}\label{32}
Let $u: {M}^k\longrightarrow \mathbb{S}^{k+1}$ be an immersed hypersurface. Then the cone
$$f (t, y, p) =(y,tu (p)): \mathbb{R}^+\times \mathbb{R}^{n-k-1}\times{M}^k\longrightarrow \mathbb{R}^{n+1}$$
is a Dupin hypersurface of constant M\"{o}bius curvatures if and only if the hypersurface $u$ is an isoparametric hypersurface in $\mathbb{S}^{k+1}$.
\end{Lemma}
\begin{proof}  From \eqref{princ} it is very clear that, if $u$ is isoparametric, then the cone $f$ is a Dupin hypersurface of constant M\"{o}bius
curvatures. To see the other direction, assume $f$ is Dupin of constant M\"{o}bius curvatures. Then, from \eqref{princ} and the fact that all M\"{o}bius curvatures
are constant, it is straightforward to see that all the ratios $\frac {\lambda_i}{\lambda_j}$ are constant. Now, one knows from the fact that $f$ is Dupin, each principal
curvature $\lambda_i$ is constant along the principal direction $e_i$ that is perpendicular to the homogeneous direction $t$. Therefore each principal curvature $\lambda_i$ is in fact constant in all directions. Thus the proof is complete.
\end{proof}

For the cone $f$ over an isoparametric hypersurface $u$ in the sphere $\mathbb{S}^{k+1}$ the eigenvalues of the Blaschke tensor and the M\"{o}bius second fundamental form are
\begin{equation}\label{cone-form}
A = diag(\underset{n-k} {\underbrace{\mu, \cdots,\mu}}, a_{n-k+1}, \cdots,a_n),  \
B = diag(\underset{n-k} {\underbrace{\lambda, \cdots, \lambda}}, b_{n-k+1}, \cdots, b_n), \text{ and }
C  =0,
\end{equation}
where
\begin{equation}\label{cone-inv}
a_\alpha  = -\lambda b_\alpha -\mu
\end{equation}
and
$$
 \mu = -\frac 1{2\rho_0^2} (1 + \frac {k^2}{n^2}H_u^2),~~ ~~\lambda =  - \frac 1{\rho_0}\frac knH_u,~~~
b_\alpha  = \frac 1{\rho_0} (\lambda_{\alpha - n+k} - \frac knH_u),
$$
for $n-k+1\leq\alpha\leq n$, following the equations \eqref{re2}. From \eqref{metric}, we know that the M\"{o}bius metric $g$ is a Riemannian product,
 that is, $(M^n,g)=(\mathbb{H}^{n-k}, \rho_0^2g_{\mathbb{H}})\times (M^k, \rho^2_0I_u)$ locally. Moreover
\begin{equation*}
B|_{T\mathbb{H}^{n-k}}=\lambda \rho_0^2g_{\mathbb{H}},~ A|_{T\mathbb{H}^{n-k}}=\mu \rho^2_0g_{\mathbb{H}}, ~ A|_{TM^k}=-\lambda B|_{TM^k}-\mu \rho_0^2I_u,
\text{ and } \lambda^2+2\mu = - \frac 1{\rho^2_0} < 0.
\end{equation*}
It is very important to observe that both the Blaschke tensor $A$ and the M\"{o}bius second fundamental form $B$ are so-called
isoparametric tensors according to Definition \ref{iso-t}. Moreover,  and equations \eqref{cone-form} and \eqref{cone-inv} are sufficient to
characterize a cone $f$ over an isoparametric surface $u$ in a sphere. Namely,

\begin{Theorem}\label{cone-type} Suppose that $f: M^n\to \mathbb{R}^{n+1}$ is a Dupin hypersurface
with $r (\geq 3)$ distinct principal curvatures and constant M\"{o}bius curvature.
And suppose that \eqref{cone-form} and \eqref{cone-inv} hold for some constants $a_\alpha, b_\alpha, \lambda, \mu$.
Then $f$ is M\"{o}bius equivalent to a cone over an isoparametric hypersurface in $\mathbb{S}^{k+1}$.
\end{Theorem}

\begin{proof} The equations \eqref{cone-form} and \eqref{cone-inv} implies that the Blaschke tensor $A$ and the M\"{o}bius second fundamental form $B$
are not linearly dependent, since $\lambda^2+2\mu<0$. From Theorem \ref{redu3} and its proof in Section \ref{proof-th1}, we know
\begin{equation}\label{bca}
R_{j\alpha j\alpha}=0 \text{ and } \omega_{j\alpha}=0 \text{ for } j =1,\cdots,n-k \text{ and } \alpha=n-k+1,\cdots,n
\end{equation}
and therefore
$$
\left\{\aligned d\omega_j & =\sum_m\omega_m\wedge\omega_{mj}=\sum_{m\leq n-k}\omega_m\wedge\omega_{mj},~~ 1 \leq j \leq n-k\\
d\omega_\beta & =\sum_m\omega_m\wedge\omega_{m\beta}=\sum_{\alpha\geq n-k+1}\omega_\alpha\wedge\omega_{\alpha\beta},~~ n-k+1\leq\beta\leq n
\endaligned\right..
$$
Thus the distributions $\mathbb{D}_1=Span\{E_{1},\cdots,E_{n-k}\}$ and $\mathbb{D}_2=Span\{E_{n-k+1},E_{n-k+2},\cdots,E_n\}$ are integrable.
Recall the corresponding M\"{o}bius position vector
$$
Y = \rho(\frac {1 + |f|^2}2, \frac{1-|f|^2}2, f): M^n\to C^{n+2}_+\subset\mathbb{R}^{n+3}_1
$$
and set
\begin{equation}\label{frame}
\begin{split}
&F=\lambda Y+\xi,~~~P=\frac{1}{\sqrt{-K}}[-(\lambda^2+\mu)Y+N- \lambda\xi],\\
&T= \frac{-1}{\sqrt{-K}}(\mu Y+N-\lambda\xi),~~~\text{and } K=\lambda^2+2\mu.
\end{split}
\end{equation}
Then $\{F,P,T,Y_1,\cdots,Y_n\}$, where $Y_i = Y_*(E_i)$, is a local orthogonal frame along $M^n$ in
$\mathbb{R}^{n+3}_1$. From the structure equations \eqref{stru} and (\ref{bca}), we get the new structure equations
\begin{equation}\label{stru1}
\aligned
& \left\{\aligned dF & = \sum_{\alpha \geq n-k+1} (\lambda- b_\alpha)\omega_\alpha Y_\alpha\\
dP & =  \sqrt{-K} \sum_{\alpha\geq n-k+1} \omega_\alpha Y_\alpha\\
dY_\alpha & = ((b_\alpha - \lambda)F - \sqrt{-K}P)\omega_\alpha  +  \sum_{\beta\geq n-k+1} \omega_{\alpha\beta}Y_\beta
\endaligned\right. \\
& \left\{\aligned
dT & =\sqrt{-K}\sum_{j=1}^{n-k} \omega_jY_j \\
dY_j & =  \omega_j \sqrt{-K}T + \sum_{k=1}^{n-k}\omega_{jk} Y_k
\endaligned\right.
\endaligned
\end{equation}
From (\ref{stru1}), we know that the subspace $V=Span\{P,Y_{n-k+1},\cdots,Y_n,F\}$
and the orthogonal complement $V^{\perp}= Span\{T,Y_{1},\cdots,Y_{n-k}\}$ are parallel along $M^{n}$. We can assume that $V=\mathbb{R}^{k+2}$ and $V^{\perp}=\mathbb{R}_1^{n-k+1}$.
Let $M^{n-k}$ be an integrable submanifold of the distribution $\mathbb{D}_1=\{E_1,E_2,\cdots,E_{n-k}\}$ and $M^{k}$
an integrable submanifold of the distribution $\mathbb{D}_2=\{E_{n-k+1},\cdots,E_n\}$. From (\ref{stru1}), we know that $P$ is constant along
the variables in $M^{n-k}$ and hence
$$P: M^k\to \mathbb{R}_1^{n+3}$$
is an $k-$dimensional immersed submanifold. Similarly, we know that
$$T: {M}^{n-k}\to \mathbb{R}_1^{n+3}$$ is an $(n-k)-$dimensional immersed submanifold.
One may calculate that $<P,P>=1$ and conclude that
$$P: M^k\to\mathbb{S}^{k+1}\subset V= \mathbb{R}^{k+2} \subset \mathbb{R}^{n+3}_1$$
since $V$ is a fixed space-like subspace. Similarly, one may calculate $<T, T> =-1$ and conclude that, up to a M\"{o}bius transformation,
$$
T: M^{n-k} \to \mathbb{H}^{n-k}\subset V^\perp = \mathbb{R}^{n-k+1}_1\subset \mathbb{R}^{n+3}_1
$$
since $V^{\perp}$ is a fixed Lorentzian subspace in $\mathbb{R}_1^{n+3}$.
In the light of \eqref{gauss1}, which is a consequence of the integrability condition \eqref{equa4}, we know that the sectional curvature for the manifold $M^{n-k}$
is, for $i,j = 1, \cdots, n-k$,
$$
R_{ijij}[g_1] = \frac 1{-K} R_{ijij} [g] = \frac 1{-K} (\lambda^2 + 2\mu) = -1,
$$
which implies that $T$ is an isometry of $\mathbb{H}^{n-k}$.

Moreover we have
$$
Y=\frac{1}{\sqrt{-K}}(T+P)
=\frac{1}{\sqrt{-K}}(T, P): M^{n-k}\times M^k\to \mathbb{H}^{n-k}\times\mathbb{S}^{k+1} \subset \mathbb{R}_1^{n+3}$$
is the M\"{o}bius position vector of the cone over an isoparametric hypersurface $P$ according to Lemma \ref{redu} and Lemma \ref{32}. Thus the proof is complete.
\end{proof}


\section{Proof of Theorem \ref{th1}}\label{proof-th1}

In this section we present the proof for Theorem \ref{th1}. By now, according to Corollary \ref{iso-type} and Theorem \ref{cone-type}, to prove Theorem \ref{th1} it
suffices to prove the following:

\begin{Theorem}\label{redu3}
Let $f:M^n\to \mathbb{R}^{n+1}$ be a Dupin hypersurface with $r (\geq 3)$ distinct principal curvatures and constant M\"{o}bius curvatures.
Then either

(a) $A$ and $B$ are linearly dependent, that is,  $A  =  \lambda B + \mu g$ for
some constants $\lambda$ and $\mu$, or

(b)  the Riemannian manifold $(M^n,g)$ is locally reducible, that is, $(M^n,g)=(M_1,g_1)\times (M_2,g_2)$ locally. Moreover
$$B|_{TM_1}=\lambda g_1,~~A|_{TM_1}= \mu g_1, ~\text{ and } ~A|_{TM_2}=-\lambda B|_{TM_2}-\mu g_2,$$
for some constants $\lambda$ and $\mu$ such that $\lambda^2 + 2\mu < 0.$
\end{Theorem}


\subsection{Outline of the proof of Theorem \ref{redu3}}

The main idea of the proof of Theorem \ref{redu3} is to use two commuting isoparametric tensors (cf. Definition \ref{iso-t}) to capture the geometric structure.
We leave the discussions on basics of isoparametric tensors including two commuting isoparametric tensors in Appendix A.
We will first show that the M\"{o}bius second fundamental form $B$ and the Blaschke tensor $A$ are commuting isoparametric tensors. Then we will show that
commuting isoparametric tensors $A$ and $B$ that satisfy the condition \eqref{equa4} will be either linearly dependent or cause the hypersurface to be reducible.

Let $f:M^n\to \mathbb{R}^{n+1}$ be a Dupin hypersurface of constant M\"{o}bius curvatures. From Theorem \ref{is1} (cf. \cite{rod}) we know that
its M\"{o}bius form $C$ vanishes and its M\"{o}bius principal curvatures $\{b_i\}$ are all constant.
Immediately from equations \eqref{equa1}, \eqref{equa2} and \eqref{equa3}, we know that the M\"obius second fundamental form $B$ and
the Blaschke tensor $A$ are two commuting Codazzi tensors. Moreover, $B$ is in fact an isoparametric tensor according to Definition \ref{iso-t}.

To make it more clearer about the behaviors of these two tensors, we can choose a local orthonormal basis $\{E_1,\cdots,E_n\}$ for $TM^n$
with respect to the M\"{o}bius metric $g$  such that
\begin{equation}\label{coe}
\begin{split}
&\left(A_{ij}\right)=diag(a_1,\cdots,a_n),\\
&\left(B_{ij}\right)=diag(b_1,\cdots,b_n)=diag(b_{\bar{1}},\cdots,b_{\bar{1}},b_{\bar{2}},\cdots,b_{\bar{2}},\cdots,b_{\bar{r}},\cdots,b_{\bar{r}}).
\end{split}
\end{equation}
Let $[i]=\{k|b_k=b_i\}$. Then $V_{b_i}=Span\{E_k|k\in [i]\}$ is the eigenspace of $B$ corresponding to the eigenvalue $b_i$.
Since $B$ is an isoparametric tensor, from \eqref{tensor1} and \eqref{tensor2}, we know
\begin{equation}\label{w2}
\left\{\begin{split}
&B_{ij,k}=0 \text{ when }~[i]=[j] ~~or~~[i]=[k], \\
&\omega_{ij}=\sum_k\frac{B_{ij,k}}{b_i-b_j}\omega_k \text{ when }~~[i]\neq [j]
\end{split}\right.
\end{equation}
and
\begin{equation}\label{cur1}
R_{ijij}=\sum_{k\notin [i],[j]}\frac{2B^2_{ij,k}}{(b_i-b_k)(b_j-b_k)} \text{ when }~~[i]\neq [j].
\end{equation}

One of the important steps in our proof is to show that the Blaschke tensor $A$ is also an isoparametric tensor. That is to show that eigenvalues
$\{a_1, \cdots, a_n\}$ are all constant according to Definition \ref{iso-t}.

\begin{Theorem}\label{blas}
Let $f:M^n\to \mathbb{R}^{n+1}$ be a M\"{o}bius isoparametric hypersurface without umbilical points. Then the eigenvalues of the Blaschke tensor
$\{a_1,\cdots,a_n\}$ are all constant.
\end{Theorem}
\begin{proof}
In the light of the classification result in \cite{lih2}, we may assume that the number $r$ of distinct principal curvatures is greater than $2$.
Since the Blaschke tensor is a Codazzi tensor, we have
\begin{equation*}
(a_i-a_j)\omega_{ij}=\sum_kA_{ij,k}\omega_k,
\end{equation*}
which implies, from (\ref{w2}),
\begin{equation}\label{aw1}
(a_i-a_j)\frac{B_{ij,k}}{b_i-b_j}=A_{ij,k} \text{ when }~~[i]\neq [j].
\end{equation}
Hence we know
\begin{equation}\label{aw2}
E_i(a_j)=A_{jj,i}=A_{ij,j}=0 \text{ when }~~ [i]\neq [j]
\end{equation}
from $B_{ij,j}=0$. Now to verify that $a_j$ is a constant, we only need to prove
\begin{equation}\label{abw}
E_i(a_j)=0,~~i\in [j].
\end{equation}
For a fixed point $p\in M^n$ and $j\in\{1, \cdots, n\}$, it is either $B_{jk,l} = 0$ for all $1\leq k,l\leq n$ or $B_{jk,l}\neq 0$ for some $1\leq k,l\leq n$.
First assume it is the second case. In fact we may assume $B_{jk,l}\neq 0$ in a neighborhood of $p$ for some $j,k,l$ that have to be associated to
three distinct M\"{o}bius principal curvatures. Therefore, from \eqref{aw1}, we obtain
\begin{equation*}
\frac{a_{j}-a_k}{b_{j}-b_k}=\frac{A_{jk,l}}{B_{jk,l}}=\frac{A_{lk,j}}{B_{lk,j}}=\frac{a_l-a_k}{b_l-b_k},
\end{equation*}
which implies
\begin{equation}\label{aw3}
a_{j}=(a_l-a_k)\frac{b_{j}-b_k}{b_l-b_k}+a_k.
\end{equation}
This easily implies (\ref{abw}). Next, suppose it is the first case. If there is a sequence of point $p_i\to p$ in $M^n$ such that the second cases happen on $p_i$ for
some $1\leq k,l\leq n$, then \eqref{abw} holds at $p$ due to the continuity. Otherwise, there is an open neighborhood $U\subset M^n$ of $p$ such that
$B_{jk,l} = 0$ for all $1\leq k,l\leq n$
in $U$. Therefore $R_{jkjk} = 0$ in $U$ from \eqref{cur1}. Hence, from \eqref{equa4}, we
derive
\begin{equation}\label{curv}
a_j = -b_jb_k - a_k \text{  in $U$ when $k\notin[j]$},
\end{equation}
which obviously implies \eqref{abw}. Thus the proof is complete.
\end{proof}

\begin{Remark}\label{codazzi-iso}
From the above proof it is clear that the following statement is true. Suppose that $A$ and $B$ are two commuting Codazzi tensors and that $B$
is an isoparametric tensor of $r(\geq 3)$ distinct eigenvalues. In addition assume $R_{ijij} = -a_i - a_j$ (cf. \eqref{2.8}).
Then $A$ is also an isoparametric tensor.
\end{Remark}

From now on in this section we will focus to studying Riemmannian manifolds with two commuting isoparametric tensors $T_1$ and $T_2$ that satisfy the condition
\begin{equation}\label{gauss1}
R_{ijkl}=\frac{1}{2}(T_1\bigodot T_1)_{ijkl}+(T_2\bigodot g)_{ijkl}
\end{equation}
according to \eqref{equa4}, where $\bigodot$ denotes the Kulkarni-Nomizu product (cf. Section \ref{subsec:iso-property-1}). We will complete the proof of Theorem \ref{redu3} in Section \ref{subsec:iso-case} and Section \ref{subsec:redu-case}. For basic properties
of isoparametric tensors readers
are referred to Section \ref{subsec:iso-property-1} and Section \ref{subsec:iso-property-2} in Appendix A.


\subsection{Linear relations of commuting isoparametric tensors}\label{subsec:iso-case}

\begin{Definition}\label{lin-iso} Let $(M^n, \ g)$ be a Riemannian manifold. Two symmetric 2-tensors $T_1$ and $T_2$ are said to be linearly dependent,
if there exist  constants $\lambda, \mu, \varepsilon $ such that $\lambda T_1+\mu T_2=\varepsilon g.$
\end{Definition}

Clearly two linearly dependent symmetric 2-tensors are always commuting. It turns out the converse is not true. In fact, two
commuting isoparametric tensors on a Riemannian manifold are not necessarily linearly dependent. In this subsection we want to give a sufficient
condition for two commuting isoparametric tensors that satisfy the \eqref{gauss1} to be linearly dependent. Our approach is to carefully study the
linear co-relations of all distinct pairs of eigenvalues of two commuting isoparametric tensors. First of all, given two commuting isoparametric tensors that satisfy \eqref{gauss1},
by Lemma \ref{ble2}, we may choose the orthonormal basis
$\{E_1,\cdots,E_n\}$ such that
\begin{equation*}
\begin{split}
\left(T_{ij}\right)=diag&(\underbrace{b_{\bar{1}},\cdots,b_{\bar{1}}},\underbrace{b_{\bar{2}},\cdots,b_{\bar{2}}},
\cdots,\underbrace{b_{\bar{r}},\cdots,b_{\bar{r}}}),\\
\left(\hat{T}_{ij}\right)=diag&(\underbrace{a_{\bar{1}},\cdots,a_{\bar{1}},\bar{a}_{\bar{1}},\cdots,\bar{a}_{\bar{1}}},\cdots,
\underbrace{a_{\bar{r}},\cdots,a_{\bar{r}},\bar{a}_{\bar{r}},\cdots,\bar{a}_{\bar{r}}}),
\end{split}
\end{equation*}
where $a_{\bar{i}}$ and $\bar{a}_{\bar{i}}$ may be same and $b_{\bar{1}}<\cdots<b_{\bar{r}}$. We then define the following two index sets
$$
[i]=\{k\in\{1, 2, \cdots, n\}|\ b_k=b_i\} \text{ and } (i)=\{k\in [i]|\ a_k=a_i\}.$$
Let $s$ be the number of the distinct groups of indices in the collection $\{(1),(2),\cdots,(n)\}$ and label these distinct
groups of indices as $\{(\bar{1}),(\bar{2}),\cdots,(\bar{s})\}$. Clearly, we have  $(i)\subseteq [i] \text{ and } s\geq r.$
For any $i\in\{1, 2, \cdots, n\}$, we consider the pair $(a_i,b_i)$ and observe that
$$
(a_i,b_i)=(a_j,b_j) \text{ if and only if } (i) = (j).
$$
Hence one may write $(a_i,b_i)= (a_{(i)},b_{(i)})$ and there are exactly $s$ distinct pairs.  Let $W$ denote the set of all of the pairs, that is,
$$W=\{(a_{(\bar{1})},b_{(\bar{1})}),(a_{(\bar{2})},b_{(\bar{2})}),\cdots,(a_{(\bar{s})},b_{(\bar{s})})\}.$$
For a number $\varepsilon$ (including $\infty$) and a group $(i)$ fixed, we  define the set of pairs
$$S_{(i)}(\varepsilon):=\{(a_k,b_k)\in W| \ \frac{a_i-a_k}{b_i-b_k}=\varepsilon,~ k\notin (i)\} \bigcup\{(a_{(i)}, b_{(i)})\}.$$
From Lemma \ref{ble2} and the above definition of $S_{(i)}(\varepsilon)$, it is easy to verify the following properties:

\begin{lemma}\label{le2} Suppose that $T_1$ and $T_2$ are two commuting isoparametric tensors on a Riemannian manifold
$(M^n, \ g)$ and satisfy the relation \eqref{gauss1}.
For a fixed index set $(i)$, the following hold: \\
(1) ${S}_{(i)}(\infty)$ can have at most two pairs;\\
(2) For two non-empty sets ${S}_{(i)}(\varepsilon_k),{S}_{(i)}(\varepsilon_l)$ and $\varepsilon_k\neq\varepsilon_l$,  ${S}_{(i)}(\varepsilon_k)\cap{S}_{(i)}(\varepsilon_l)=\{(a_{(i)}, b_{(i)})\}$;\\
(3) There exist only finitely many constants (including $\infty$) $\varepsilon_1,\cdots,\varepsilon_l$ such that
${S}_{(i)}(\varepsilon_1), \cdots, {S}_{(i)}(\varepsilon_l) $ are non-empty;\\
(4) If  the set ${S}_{(i)}(\varepsilon)= \{(a_{(i)},b_{(i)}), (a_{(j)},b_{(j)})\}$ for $j\notin(i)$, then
\begin{equation}\label{abco1}
R_{klkl}=b_ib_j+a_i+a_j=0 \text{ for all $k\in(i)$ and $l\in(j)$.}
\end{equation}
\end{lemma}
\begin{proof} These properties are all trivial except (4). It suffices to show that
$T_{kl,m}=0$ for all $m = 1, 2, \cdots, n$ when $k\in(i)$ and $l\in(j)$. The nontrivial cases are $k\in(i)\subset[i]$, $l\notin [i]$ and $m\notin [i]\cup[l]$.
Hence, from  the third equation in (\ref{cotensor}), we would have
$$
\frac{a_{m}-a_k}{b_{m}-b_k}=\frac{\hat{T}_{km,l}}{T_{km,l}}=\frac{\hat{T}_{lk,m}}{T_{lk,m}}=\frac{a_{l}-a_k}{b_{l}-b_k}
$$
if $T_{kl, m} = T_{km,l}$ were not vanishing.  That would imply $(a_m, b_m)\in S_{(i)}(\varepsilon)$ and a contradiction to assumption that $S_{(i)}(\varepsilon)$ has
only two pairs.  Thus the proof is complete.
\end{proof}

Next we want to understand the geometric impacts for the set $S_{(i)}(\varepsilon)$ to contain more than two pairs. Again the key is to establish the generalized
Cartan identity that relates the sectional curvatures in the planes generated by the eigenvectors whose eigenvalues lie in the set $S_{(i)}(\varepsilon)$.

\begin{lemma}\label{le31} Suppose that $T_1$ and $T_2$ are two commuting isoparametric tensors on a Riemannian manifold
$(M^n, g)$ and satisfy the relation \eqref{gauss1}.
And suppose that, for some $i$ and $\varepsilon$,
$${S}_{(i)}(\varepsilon)=\{(a_{i_1},b_{i_1}),(a_{i_2},b_{i_2}),\cdots,(a_{i_t},b_{i_t})\}$$
has $t$ number of distinct pairs for $t\geq 3$.
Then, for $j\in (i_k)$, $h\in(i_l)$, $i_k\neq i_l$,
\begin{equation}\label{curpair}
R_{jhjh}=\sum_{m\in (i_1)\cup(i_2)\cup\cdots\cup(i_t),
 m\notin (i_k)\cup(i_l)}\frac{T^2_{jh,m}}{(b_{i_k}-b_m)(b_{i_l}-b_m)}.
 \end{equation}
 Consequently, if $b_{i_1}<b_{i_2}<\cdots<b_{i_t}$, then
\begin{equation}\label{abco2}
\begin{split}
&R_{ijij} = b_{i_m}b_{i_{m+1}} + a_{i_m} + a_{i_{m+1}} \geq 0,~~~i\in(i_m),~j\in(i_{m+1}),~~~m=1,2,\cdots,t-1,\\
&R_{ijij} = b_{i_1}b_{i_t} + a_{i_1} + a_{i_t} \leq0,~~~i\in(i_1),~j\in(i_t).
\end{split}
\end{equation}
More importantly, for each $i_k$ fixed and $j\in (i_k)$, we have the following generalized Cartan identity,
\begin{equation}\label{line-cartan}
\sum_{m\in (i_1)\cup(i_2)\cup\cdots\cup(i_t), m\notin (i_k)}\frac{R_{jmjm}}{b_{i_k}-b_m}=0.
\end{equation}
\end{lemma}
\begin{proof} It suffices to prove that
$$
T_{jh,m}=0 \text{ when } j,h\in (i_1),(i_2),\cdots,(i_t), \text{ and } m\notin (i_1),(i_2),\cdots,(i_t).
$$
The argument is exactly the same as the proof of (4) in the above lemma. Because $T_{jh,m}\neq 0$ would imply that
$(a_m, b_m)\in S_{(i)}(\varepsilon)$, which is a contradiction.
\end{proof}

Consequently we observe

\begin{lemma}\label{le5} Under the same assumptions as in Lemma \ref{le31},
$$b_{i_1}+\varepsilon< 0,~~~b_{i_t}+\varepsilon> 0.$$
\end{lemma}
\begin{proof}
From Lemma \ref{le31}, we have
\begin{equation}\label{start-end}
R_{i_1i_2i_1i_2}\geq 0,~~R_{i_1i_ti_1i_t}\leq 0.
\end{equation}
In the light of the assumption (\ref{gauss1}), we arrive at
\begin{equation*}
\begin{split}
0\leq R_{i_1i_2i_1i_2}-R_{i_1i_ti_1i_t}=b_{i_1}(b_{i_2}-b_{i_t})+a_{i_2}-a_{i_t}\\
=(b_{i_2}-b_{i_t})(b_{i_1}+\frac{a_{i_2}-a_{i_t}}{b_{i_2}-b_{i_t}})=(b_{i_2}-b_{i_t})(b_{i_1}+\varepsilon),
\end{split}
\end{equation*}
which implies $b_{i_1}+\varepsilon\leq 0.$ To prove that $b_{i_1}+\varepsilon< 0$ we do it by contradiction. We assume otherwise
$b_{i_1}+\varepsilon= 0$ and hence $b_{i_t} + \varepsilon >b_{i_1}+\varepsilon=0$. Immediately we have
$$
R_{i_1ki_1k}-R_{i_1li_1l}=b_{i_1}(b_k-b_l)+a_k-a_l=(b_k-b_l)(b_{i_1}+\varepsilon)=0 \text{ for any } k\in (i_k) \text{ and } l\in(i_l).
$$
That is
$$
R_{i_1ki_1k}=R_{i_1li_1l} \text{ for any } k\in (i_k) \text{ and } l\in(i_l),
$$
which forces, from \eqref{start-end}, $R_{i_1ki_1k}=0$ for any $k\in (i_2)\cup(i_3)\cdots\cup(i_t)$. On the other hand,
$$R_{i_tki_tk}-R_{i_tli_tl}=b_{i_t}(b_k-b_l)+a_k-a_l=(b_k-b_l)(b_{i_t}+\varepsilon)~~ \text{ for any } k\in (i_k) \text{ and } l\in(i_l),$$
which implies
$$
0=R_{i_ti_1i_ti_1} < R_{i_ti_{i_2}i_ti_{i_2}} < \cdots < R_{i_ti_{t-1}i_ti_{t-1}}
$$
and hence
$$
R_{i_tmi_tm}> 0 \text{ for any } m\in (i_2)\cup(i_3)\cdots\cup(i_{t-1}).
$$
Therefore,  in the light of the generalized Cartan identity \eqref{line-cartan} for $i=i_t$,
\begin{equation*}
\sum_{m\in (i_2)\cup(i_3)\cup\cdots\cup(i_{t-1})}\frac{R_{i_tmi_tm}}{b_{i_t}-b_m}=0
\end{equation*}
and therefore $R_{i_tmi_tm} = 0$ for $m\in(i_1)\cup(i_2)\cup\cdots\cup(i_{t-1})$ and $b_{i_t}+\varepsilon=0$, which is a contradiction when
$t\geq 3$.

To prove $b_{i_t} + \varepsilon > 0$, similarly from Lemma \ref{le31}, we start with
$$R_{i_ri_{t-1}i_ti_{t-1}}\geq 0 \text{ and }R_{i_1i_ti_1i_t}\leq 0$$
and
$$
0 \leq R_{i_ti_{t-1}i_ti_{t-1}}-R_{i_1i_ti_1i_t}=(b_{i_{t-1}}-b_{i_1})(b_{i_t}+\varepsilon),
$$
to conclude $b_{i_t}+\varepsilon\geq 0$. Then, with the argument similar as that in the above, we can derive
a contradiction if $b_{i_t}+\varepsilon = 0$. Thus the proof is complete.
\end{proof}

The following is another technical lemma we will use to discover the structure of the distribution of pairs $(a_i, b_i)$
in the plane and facilitate the proof of Theorem \ref{th1}. It is clear that pairs in each set $S_{(i)}(\varepsilon)$ lie in
a line with the slope $\varepsilon$ if $S_{(i)}(\varepsilon)$ has more than one pairs. We observe

\begin{lemma}\label{le4} Under the same assumptions as that in Lemma \ref{le31}, if the set $S_{(i)}(\varepsilon)$ has at least three distinct pairs and
the line equation for the set $S_{(i)}(\varepsilon)$ is
$$
a=\varepsilon b+d
$$
for some constant $d$, then
\begin{equation}\label{slope-intercept}
\varepsilon^2 - 2d < 0.
\end{equation}
\end{lemma}
\begin{proof} As in the above, let $S_{(i)}(\varepsilon)=\{(a_{i_1},b_{i_1}),(a_{i_2}, b_{i_2}),\cdots,(a_{i_t},b_{i_t})\}$ and $b_{i_1}<b_{i_2} < \cdots<b_{i_t}$.
By the assumption, each pair $(a_{i_k}, b_{i_k})$ in $S_{(i)}(\varepsilon)$ satisfies the line equation
$$
a_{i_k} = \varepsilon b_{i_k} + d.
$$
Hence, for $i\in (i_k)$ and $j\in(i_l)$, from the assumption \eqref{gauss1}, we obtain
\begin{equation}\label{cur-line}
R_{ijij}=b_ib_j+a_i+a_j=(b_{i_k}+\varepsilon)(b_{i_l}+\varepsilon)+2d-\varepsilon^2.
\end{equation}
For the simplicity, we will use the notations $\tilde b_i = b_i +\varepsilon$ in the following.

We first claim that $2d - \varepsilon \geq 0$. Assume otherwise that $2d - \varepsilon < 0$.
From Lemma \ref{le31}, we know that
$$
0 \leq R_{i_mi_{m+1}i_mi_{m+1}} = \tilde b_{i_m}\tilde b_{i_{m+1}} + 2d - \varepsilon^2,~~m=1,2,cdots, t-1,
$$
which implies
\begin{equation}\label{transitive}
\tilde b_{i_m}\tilde b_{i_{m+1}} > 0,
\end{equation}
for all $m\in\{1, 2, \cdots, t-1\}$.
Under the assumptions, we know from Lemma \ref{le5} that $\tilde b_{i_1}
<0$. Therefore we may conclude that $\tilde b_{i_t} < 0$ in the light of \eqref{transitive}, which is
a contradiction to $\tilde b_{i_t} > 0$ in Lemma \ref{le5}. So we have $2d- \varepsilon \geq 0$.

Next we want to exclude the cases that $2d - \varepsilon = 0$. Assume again otherwise that
$2d - \varepsilon = 0$. From the generalized Cartan identity for $i=i_k$ in Lemma \ref{le31}, we write
$$
0 = \sum_{m\in (i),(i_1),(i_2),\cdots,(i_t), m\notin (i_k)}
\frac{\tilde b_{i_k}\tilde b_m}{\tilde b_{i_k}-\tilde b_m}=\sum_{m\in (i),(i_1),(i_2),\cdots,(i_t), m\notin (i_k)}
\frac 1 {\frac 1{\tilde b_{i_m}} - \frac 1 {\tilde b_{i_k}}}.
$$
This is impossible if one takes $i_k = i_1$ or $i_k = i_t$. Thus $2d-\varepsilon^2>0$.
\end{proof}

It is obvious that, in order for tensors $T_1$ and $T_2$ to be linearly dependent, all the pairs have to be lined in one set $S_{(i)}(\varepsilon)$. Particularly it is necessary that
$(i) = [i]$ for each $i=1, 2, \cdots, n$, i.e. $r=s$. Let us list all the sets $S_{(\bar 1)}(\varepsilon)$ which have more than one pairs
$$
S_{(\bar 1)}(\varepsilon_1), S_{(\bar 1)}(\varepsilon_2), \cdots, S_{(\bar 1)}(\varepsilon_t),
$$
where $\varepsilon_1 < \varepsilon_2 < \cdots < \varepsilon_t$. Now we are ready to state a theorem for linear dependence of two commuting isoparametric tensors.

\begin{Theorem}\label{case2} Suppose that $T_1$ and $T_2$ are two commuting isoparametric tensors on a Riemannian manifold
$(M^n, \ g)$ and satisfy the relation \eqref{gauss1}. And suppose that $r=s$. If the set $S_{(\bar{1})}(\varepsilon_1)$ has at least three distinct pairs, or $r=s\leq 2$,
then there exists constant $\mu$ such that
\begin{equation}\label{lin-dep}
T_2=\varepsilon_1 T_1+\mu g.
\end{equation}
\end{Theorem}

Before the proof of Theorem \ref{case2}, we first establish a sequence of lemmas.  Let
$$S_{(\bar{1})}(\varepsilon_1)=\{(a_1,b_1),(a_{i_1},b_{i_1}),\cdots,(a_{i_k},b_{i_k})\}
$$ and $b_1<b_{i_1}<\cdots<b_{i_k}$ for
some $k\geq 2$.

\begin{lemma}\label{case2-1} Under the assumptions in Theorem \ref{case2},
\begin{equation}\label{eq-case2-1}
b_{\bar 1} +\varepsilon_i < 0
\end{equation}
for all $i=1, 2, \cdots, t$.
\end{lemma}

\begin{proof} In the light of Lemma \ref{le5} one may only need to consider the cases when $S_{(\bar 1)}(\varepsilon_i) $ has exactly two pairs, say,
$S_{(\bar 1)}(\varepsilon_i) =\{(a_1, b_1), (a_j, b_j)\}$ for some
$i=2, 3, \cdots, t$. Hence, from Lemma \ref{le31}, we know that
$$
b_1b_j +a_1 +a_j = 0.
$$
On the other hand,
$$b_1^2+2a_1-(b_1b_{i_1}+a_1+a_{i_1})=(b_1-b_{i_1})(b_1+\frac{a_1-a_{i_1}}{b_1-b_{i_1}})=(b_1-b_{i_1})(b_1+\varepsilon_1)>0.
$$
Since $b_{1} +\varepsilon_1 <0$ due to Lemma \ref{le5}. Therefore
$$
b_1^2 + 2a_1 > b_1b_{i_1} + a_1 + a_{i_1} = R_{1i_11i_1} \geq 0
$$
by Lemma \ref{le31} again. Thus
$$
0 < b_1^2 +2a_1 -(b_1b_j +a_1+a_j) = (b_1 - b_j)(b_1 + \varepsilon_i),
$$
which implies that $b_1 +\varepsilon_i < 0$. So the proof is complete.
\end{proof}

To set some notations, let the line equation for each set $S_{(\bar 1)}(\varepsilon_m)$ is
$$
a = \varepsilon_m b + d_m,~~m=1,2,\cdots, t.
$$
Then we have
\begin{equation}\label{aba}
- b_{\bar 1} = \frac {d_i - d_j}{\varepsilon_i - \varepsilon_j}
\end{equation}
for all $i,j = 1, 2, \cdots, t$.

\begin{lemma}\label{case2-2} Under the assumptions in Theorem \ref{case2},
\begin{equation}\label{eq:case2-2}
b_j^2 +2a_i > 0
\end{equation}
for all $j=1, 2, \cdots n$.
\end{lemma}
\begin{proof} First, for a given pair $(a_j, b_j)$, it lies in $S_{(\bar 1)}(\varepsilon_{m})$ for some $m=1, 2, \cdots, t$. Then
$$
 b_j^2 + 2a_j = (b_j + \varepsilon_{m})^2 - \varepsilon_{m}^2 + 2d_{m}.
$$
To prove $b_j^2 +2a_j >0$ one may verify that $-\varepsilon_{m}^2 + 2d_{m}>0$.
From Lemma \ref{le4}, we know that $ - \varepsilon^2_1+ 2d_1 >0$. Hence we use \eqref{aba} to calculate
$$
-\varepsilon_{m}^2 + 2d_{m} - (- \varepsilon^2_1+ 2d_1) = (\varepsilon_1 - \varepsilon_{m})( 2b_1 + \varepsilon_1 + \varepsilon_{m}),
$$
which is positive according to Lemma \ref{case2-1}. Therefore $-\varepsilon_{m}^2 + 2d_{m}>0$ and the proof is complete.
\end{proof}
The following lemma is very useful to understand what are the possible lines that connect all pairs in $W$.
Notice that, if $S_{(\bar 1)}(\varepsilon_1)$ does not contain all the pairs, then the set $S_{(i_l)}(\varepsilon)$ for some $\varepsilon\neq\varepsilon_1$
must contain more than one pairs.

\begin{lemma}\label{case2-3-4} Assume the same assumptions of Theorem \ref{case2} hold. Then, for any set $S_{(i_l)}(\varepsilon), l=1,\cdots,k,$ that contains more than
one pairs, if $\varepsilon < \varepsilon_1$, then $b_{i_l} + \varepsilon >0$
and $b_{i_l} \geq b_j$ for all $(a_j, b_j)\in S_{({i_l})}(\varepsilon)$.
Similarly, if $\varepsilon > \varepsilon_1$, then $b_{i_l} + \varepsilon < 0$
and $b_{i_l} \leq b_j$ for all $(a_j, b_j)\in S_{({i_l})}(\varepsilon)$.
\end{lemma}

\begin{proof} Let $d_j= a_j - \varepsilon_1b_j$ for $j =1, 2, \cdots, n$. Then it is clear that
$d_j=d_1$ when $j\in [1]\cup[i_1]\cup\cdots\cup[i_k]$. On the other hand, if $j\notin [1]\cup[i_1]\cdots\cup[i_k]$, then
$$
d_1-d_j=(b_1-b_j)(-\varepsilon_1+\frac{a_1-a_j}{b_1-b_j}) <0.
$$
Because $\varepsilon_1$ is the smallest slope among all lines passing through $(a_1, b_1)$. Therefore
\begin{equation}\label{1db1}
d_1=d_j=min\{d_1,d_2,\cdots,d_n\},~~j\in [1]\cup[i_1]\cup\cdots\cup[i_k].
\end{equation}
First, we consider $\varepsilon < \varepsilon_1$.  for any $(a_j,b_j)\in {S}_{({i_l})}(\varepsilon)\setminus\{(a_{i_l}, b_{i_l})\}$, we have
\begin{equation*}
d_j-d_{i_l}=(b_j-b_{i_l})(\frac{a_{i_l}-a_j}{b_{i_l}-b_j}-\varepsilon_1)=(b_j-b_{i_l})(\varepsilon-\varepsilon_1).
\end{equation*}
Thus, by (\ref{1db1}), we see that $b_j\leq b_{i_l}$.

Again, to show $b_{i_l}+\varepsilon >0$, in the light of Lemma \ref{le5}, one may assume that $S_{({i_l})}(\varepsilon)$ has exactly two pairs, say,
$S_{(\bar{i_l})}(\varepsilon) = \{(a_j, b_j), (a_{i_l}, b_{i_l})\}$. From Lemma \ref{le31}, we know
$$R_{ji_lji_l}=b_jb_{i_l}+a_j+a_{i_l}=0.
$$
Similar to the proof of Lemma \ref{case2-1}, we calculate
$$
b_{i_l}^2  + 2a_{i_l} = b_{i_l}^2  + 2a_{i_l} - (b_jb_{i_l}+a_j+a_{i_l}) = (b_{i_l} - b_j)(b_{i_l} + \varepsilon),
$$
which implies $b_{i_l}+\varepsilon > 0$ due to Lemma \ref{case2-2}. The above proof works for the case $\varepsilon > \varepsilon_1$.
Thus the proof is complete.
\end{proof}

We now are ready to present the proof for Theorem \ref{case2}.

\begin{proof}[Proof of Theorem \ref{case2}] For the case $r=s\leq 2$, obviously $T_2=\varepsilon_1 T_1+\mu g$ for some constant $\mu$.
Next let $r=s\geq 2$.
We assume otherwise that $S_{(\bar 1)}(\varepsilon_1)$ does not contain all the pairs. Recall
$$S_{(\bar{1})}(\varepsilon_1)=\{(a_1,b_1),(a_{i_1},b_{i_1}),\cdots,(a_{i_k},b_{i_k})\}
$$ and $b_1<b_{i_1}<\cdots<b_{i_k}$ for
some $k\geq 2$. For each $i_l\in \{i_1,\cdots,i_k\}$ fixed, let
$${S}_{(i_l)}(\varepsilon_1), {S}_{(i_l)}(\varepsilon_{j_1}^l),\cdots, {S}_{(i_l)}(\varepsilon_{j_h}^l)$$
be the full list of the set that contain more than just the pair $(a_{i_l}, b_{i_l})$ and
$\varepsilon_{j_1}^l<\cdots<\varepsilon_{j_h}^l$, where $h\geq 1$ and may depend of $l$. Our argument is to show there is no way
to compare $\varepsilon_1$ with the rest slopes $\varepsilon^l_{j_m}$ when $l=k$.

It is easy to see that $\varepsilon_1$ has to be large than each slope $\varepsilon^k_{j_m}$. Assume otherwise $\varepsilon_1 < \varepsilon^k_{j_m}$ for some $m$.
Then, on one hand, applying Lemma \ref{le5} to $S_{(\bar 1)}(\varepsilon_1)$, we obtain
$$
b_{i_k} + \varepsilon_1 >0.
$$
On the other hand, applying Lemma \ref{case2-3-4} to $S_{(i_k)}(\varepsilon^k_{j_m})$, we obtain
$$
b_{i_k} + \varepsilon_1 < b_{i_k} + \varepsilon^k_{j_m} < 0,
$$
which is a contradiction.

Next we want to show that it is also impossible to have $\varepsilon_1$ larger than each slope $\varepsilon^k_{j_m}$. But, first, we can easily see that $\varepsilon_1$
can not lie in between two slopes $\varepsilon^l_{j_m}$ for any $l=1, 2, \cdots, k$. Because, if it happened, then we would have from Lemma \ref{case2-3-4} that
$$
b_{i_l} + \varepsilon^l_{j_{m-1}} > 0 \text{ and } b_{i_l} + \varepsilon^l_{j_m} < 0,
$$
which is a contradiction. To finish the proof we show inductively for $l=1, 2, \cdots, k$ that it is impossible that
$$
\varepsilon^l_{j_1}<\varepsilon^l_{{j_2}}<\cdots<\varepsilon^l_{j_h}<\varepsilon_1.
$$

Before we start the induction, we define
$$b_{j_0}=min\{b_j| ~(a_j,b_j)\notin {S}_{(\bar 1)}(\varepsilon_1)\}.
$$
When $l=1$,  if $\varepsilon_1$ is larger than every other slope $\varepsilon^1_{j_m}$, then using Lemma \ref{case2-3-4},  we know $b_{i_1}$ and $b_{j_0}$
are the largest and the smallest for pairs in the line that connects $(a_{j_0}, b_{j_0})$ to $(a_{i_1}, b_{i_1})$. Therefore, by Lemma \ref{le31},
$$R_{1i_11i_1}=b_1b_{i_1}+a_1+a_{i_1}\geq 0,~~R_{i_1j_0i_1j_0}=b_{i_1}b_{j_0}+a_{i_1}+a_{j_0}\leq 0,$$
which implies
$0\leq R_{1i_11i_1}-R_{i_1i_0i_1i_0}=(b_1-b_{j_0})(b_{i_1}+\frac{a_1-a_{j_0}}{b_1-b_{j_0}})$
and hence $$b_{i_1}+\frac{a_1-a_{j_0}}{b_1-b_{j_0}}\leq 0.$$
On the other hand, from Lemma \ref{case2-3-4}, we have
$$
0 < b_{i_1} + \varepsilon^1_{j_m} <b_{i_1} + \frac{a_1-a_{j_0}}{b_1-b_{j_0}}
$$
which is a contradiction. Here we used the fact that $\varepsilon_1 < \frac{a_1-a_{j_0}}{b_1-b_{j_0}}$.
So $\varepsilon_1$ can only be smaller than every other slope $\varepsilon^1_{j_m}$
and
\begin{equation}\label{ind-1}
\varepsilon_1 < \frac {a_{i_1} - a_{j_0}}{b_{i_1} - b_{j_0}} \text{ and } b_{i_1} < b_{j_0}.
\end{equation}
Similarly, from Lemma \ref{le31},
$$
R_{i_1i_2i_1i_2}  \geq 0 \text{ and } R_{i_2j_0i_2j_0} \leq 0,
$$
which implies $0\leq R_{i_1i_2i_1i_2}-R_{i_2j_0i_2j_0}=(b_{i_1}-b_{j_0})(b_{i_2}+\frac{a_{i_1}-a_{j_0}}{b_{i_1}-b_{j_0}})$
and $$b_{i_2}+\frac{a_{i_1}-a_{j_0}}{b_{i_1}-b_{j_0}}\leq 0.
$$
If otherwise $\varepsilon_1$ is larger than every slope $\varepsilon^2_{j_m}$, then, from Lemma \ref{case2-3-4},
$$
0 < b_{i_2} + \varepsilon^2_{j_m}  < b_{i_2} + \varepsilon_1 < b_{i_2} + \frac{a_{i_1}-a_{j_0}}{b_{i_1}-b_{j_0}},
$$
which is a contradiction and implies
\begin{equation}\label{ind-1}
\varepsilon_1 < \frac {a_{i_2} - a_{j_0}}{b_{i_2} - b_{j_0}} \text{ and } b_{i_2} < b_{j_0}.
\end{equation}
By induction, we can prove $\varepsilon_1$ cannot be larger than every slope $\varepsilon^k_{j_m}$. Thus the proof is finished.
\end{proof}


\subsection{Reducible cases}\label{subsec:redu-case}

In this subsection we want to show that, if the assumptions in Theorem \ref{case2} for two commuting isoparametric tensors that satisfy \eqref{gauss1} are not true, then
the underlined Riemannain manifold has to be reducible. The first cases are when $r< s$, that is, when $T_2$ restricted to some eigensapce $V_{b_i}$ has two distinct
eigenvalues. The other cases are when the set $S_{(\bar 1)}(\varepsilon_1)$ has two distinct pairs.

Let us deal with the first cases. In the light of Lemma \ref{ble2}, we may assume that for some $\bar k = \bar 1, \bar 2, \cdots, \bar r$
\begin{equation}\label{split-k}
[\bar k] = (k_1) \bigcup (k_2) \text{ and } a_{k_1} < a_{k_2}.
\end{equation}

\begin{lemma} \label{case1-1} Suppose that $T_1$ and $T_2$ are two commuting isoparametric tensors on a Riemannian manifold $(M^n, \ g)$ that satisfy the condition \eqref{gauss1}.
And suppose that \eqref{split-k} holds for some $\bar k = \bar 1, \bar 2, \cdots, \bar r$. Then any set $S_{(k_1)}(\varepsilon)$ has at most two pairs.
\end{lemma}

\begin{proof} We assume otherwise that the ${S}_{(k_1)}(\varepsilon)$ has at least three distinct pairs for some $\varepsilon$, say,
$$
{S}_{(k_1)}(\varepsilon)=\{(a_{k_1},b_{k_1}),(a_{j_1},b_{j_1}),\cdots,(a_{j_t},b_{j_t})\}$$
for some $t\geq 2$, where $b_{j_1} < b_{j_2}<\cdots<b_{j_t}$. Our argument is again to show that there is no
way to compare $b_{k_1}$ with the rest $b_{j_1}, b_{j_2}, \cdots, b_{j_t}$.

First we want to show that $b_{k_1}$ cannot be in between the rest. Assume otherwise that for some $l=1, 2, \cdots, t-1$
$$
b_{j_l} < b_{k_1} < b_{j_{l+1}}.
$$
From Lemma \ref{le31}, we know that
$$
b_{j_l}b_{k_1} + a_{j_l} + a_{k_1}\geq 0 \text{ and } b_{j_{l+1}}b_{k_1} + a_{j_{l+1}} + a_{k_1}\geq 0.
$$
And, from Lemma \ref{ble2}, we know that $b_{k_1}^2 + a_{k_1} + a_{k_2} = 0.$ Hence, on one hand,
\begin{equation*}
\begin{split}
0 \leq &b_{j_{l+1}}b_{k_1} + a_{j_{l+1}} + a_{k_1} - (b_{k_1}^2 + a_{k_1} + a_{k_2} )\\
&= (b_{j_{l+1}} - b_{k_1})(b_{k_1} + \frac{a_{j_{l+1}} - a_{k_2}}{b_{j_{l+1}} - b_{k_2}})= (b_{j_{l+1}} - b_{k_1})(b_{\bar k} + \varepsilon + \frac {a_{k_1} - a_{k_2}}{b_{j_{l+1}} - b_{\bar k}}),
\end{split}
\end{equation*}
which implies
$$
b_{\bar k} + \varepsilon > b_{\bar k} + \varepsilon + \frac {a_{k_1} - a_{k_2}}{b_{j_{l+1}} - b_{\bar k}} \geq 0.
$$
On the other hand,
$$
0 \leq b_{j_{l}}b_{k_1} + a_{j_{l}} + a_{k_1} - (b_{k_1}^2 + a_{k_1} + a_{k_2} ) = (b_{j_{l}} - b_{k_1})(b_{k_1} + \frac{a_{j_{l}} - a_{k_2}}{b_{j_{l}} - b_{k_2}}),
$$
which implies
$$
b_{\bar k} + \varepsilon < b_{\bar k} + \varepsilon + \frac {a_{k_1} - a_{k_2}}{b_{j_{l}} - b_{\bar k}} \leq 0,
$$
which is a contradiction.

Next we want to show that $b_{k_1}$ cannot be smaller than all the rest. Assume otherwise
$$
b_{k_1} < b_{j_1} < b_{j_2} < \cdots < b_{j_t}.
$$
From Lemma \ref{le5}, we know that $b_{\bar{k}}+\varepsilon<0$. But, on the other hand, applying Lemma \ref{le31}, we know that
$$
0 \leq b_{k_1}b_{j_1} + a_{k_1} + a_{j_1}  - (b_{k_1}^2 + a_{k_1} + a_{k_2} ) = (b_{j_1} - b_{k_1})(b_{k_1} + \frac{a_{j_{1}} - a_{k_2}}{b_{j_{1}} - b_{k_2}}),
$$
which implies
$$
b_{\bar k} + \varepsilon > b_{\bar k} + \varepsilon + \frac {a_{k_1} - a_{k_2}}{b_{j_{1}} - b_{\bar k}} \geq 0,
$$
which is a contradiction. One may find similarly $b_{\bar k}$ cannot be larger than the rest. Thus the proof is complete.
\end{proof}

Now we are ready to solve the cases when $s>r$.

\begin{Theorem}\label{case1} Suppose that $T_1$ and $T_2$ are two commuting isoparametric tensors on a Riemannian manifold $(M^n, \ g)$
that satisfy the condition \eqref{gauss1}. And suppose that there exists an eigenspace $V_{b_{\bar{k}}}$ of $T_1$
such that $T_2|_{V_{b_{\bar{k}}}}$ has two distinct eigenvalues $a_{k_1} < a_{k_2}$. Then
\begin{equation}\label{neg}
b_{\bar{k}}^2+2a_{k_1}<0
\end{equation}
and the Riemannian manifold $(M^n,g)$ is locally reducible, that is, $(M^n,g)=(M_1,g_1)\times (M_2,g_2)$ locally. Moreover
$$T_1|_{TM_1}=b_{\bar{k}}g_1,~~T_2|_{TM_1}=a_{k_1}g_1 \text{ and } T_2|_{TM_2}=-b_{\bar{k}} T_1|_{TM_2}-a_{k_1}g_2.$$
\end{Theorem}

\begin{proof} In the light of Lemma \ref{case1-1} we know each set $\hat{S}_{(k_1)}(\varepsilon)$ has at most two pairs, where $[\bar k] = (k_1)\cup(k_2)$.
Then, from  Lemma \ref{le2}, we know
$$
R_{ijij} = b_{\bar k}b_j + a_{k_1} + a_j =0 \text{ for all } i\in (k_1) \text{ and } j\notin (k_1)
$$
Therefore, for $j\notin [\bar k]$, we calculate from
$$
0 = b_{\bar{k}}b_j+a_{k_1}+a_j-(b_{\bar{k}}^2+a_{k_1}+a_{k_2})=(b_j-b_{\bar{k}})(b_{\bar{k}}+\frac{a_j-a_{k_2}}{b_j-b_{\bar{k}}}).
$$
that
$$
\frac{a_j-a_{k_2}}{b_j-b_{\bar{k}}}=-b_{\bar{k}}.
$$
Thus each pair $(a_j, b_j)$ for any $j\notin (k_1)$ falls in ${S}_{(k_2)}(-b_{\bar{k}})$ and satisfies
$$
a_j=-b_{\bar{k}}b_j-a_{k_1}.
$$

To finish the proof we only need to verify that both the distribution $V_{a_{k_1}} = span\{E_i| \ i\in (k_1)\}$ and its orthogonal compliment $V^\perp
= span\{E_i| \ i\notin (k_1)\}$ are integrable and parallel. According to \cite{sn}, that amounts to show that $\omega_{ij}=0$ for all $i\in (k_1)$ and $j\notin (k_1)$.
We first verify that, for $i\in (k_1)$ and $j\in (k_2)$, $\omega_{ij}=0$. This is because, using \eqref{aa1}, one only needs to see that
$\hat T_{ij,l} = 0$ when $l\in [\bar k] = (k_1)\cup(k_2)$ from the second equation in \eqref{cotensor}.
We then claim $\hat{T}_{ij,m}=0$ for all $i\in (k_1), j\notin [\bar k[, m = 1, 2, \cdots, n$. We are proving this claim by repeatedly using the Cartan identity \eqref{tensor2}
and \eqref{abco1} in Lemma \ref{le2}
\begin{equation}\label{curv0}
0=R_{ijij}=\sum_{m\notin [\bar{k}],[j]}\frac{2T^2_{ij,m}}{(b_m-b_{\bar{k}})(b_m-b_j)},~~when~~i\in (k_1)\subset [\bar{k}],~~j\notin [\bar{k}].
\end{equation}
First, let $j\in [\overline{k+1}]$ in (\ref{curv0}), and we note that $$(b_m-b_{\bar{k}})(b_m-b_j)>0,~~ when ~~m\notin [\bar{k}],[\overline{k+1}],$$
which forces
\begin{equation}\label{curv1}
T_{ij,m}=0,~~when~~i\in (k_1),~j\in [\overline{k+1}],~~1\leq m\leq n.
\end{equation}
Then let $j\in [\overline{k+2}]$ in (\ref{curv0}) and obtain
\begin{equation*}
0 =\sum_{m\notin [\bar{k}],[\overline{k+2}]}\frac{2T^2_{ij,m}}{(b_m-b_{\bar{k}})(b_m-b_{\overline{k+2}})}
=\sum_{m\notin [\bar{k}],[\overline{k+1}],[\overline{k+2}]}\frac{2T^2_{ij,m}}{(b_m-b_{\bar{k}})(b_m-b_{\overline{k+2}})}
\end{equation*}
due to \eqref{curv1}, which in turn improves \eqref{curv1} into
$$
T_{ij,m}=0,~~when~~i\in (k_1),~j\in [\overline{k+1}]\cup[\overline{k+2}],~~1\leq m\leq n.
$$
Repeatedly extending in both directions we can prove the claim $T_{ij,m}=0$ for all $i\in (k_1),~j\notin [\bar{k}]$ and all $1\leq m\leq n.$
Therefore, from (\ref{tensor1}), $\omega_{ij}=0$ for all $i\in (k_1)$ and $j\notin [\bar{k}]$. So the proof is complete.
\end{proof}

The rest cases are when $r=s$ but $S_{(\bar 1)}(\varepsilon_1)$ has two pairs. Let
\begin{equation}\label{=2}
S_{(\bar 1)}(\varepsilon_1) = \{(a_{\bar 1}, b_{\bar 1}), (a_{\bar k}, b_{\bar k})\}.
\end{equation}
The patterns of correlations among all the pairs are exactly the same as what they
are in the cases of Theorem \ref{case1}, that is, every $S_{(\bar k)}(\varepsilon)$ has at most two pairs and therefore all pairs except $(a_{\bar k}, b_{\bar k})$
lie in one line.

\begin{Theorem}\label{case3} Suppose that $T_1$ and $T_2$ are two commuting isoparametric tensors on a Riemannian manifold
$(M^n, \ g)$ and satisfy the relation \eqref{gauss1}. And suppose that $r=s> 2$ and \eqref{=2} holds. Then
\begin{equation}\label{neg}
b_{\bar k}^2+2a_{\bar k}<0
\end{equation}
and the Riemannian manifold $(M^n,g)$ is locally reducible, that is, $(M^n,g)=(M_1,g_1)\times (M_2,g_2)$ locally. Moreover
$$
T_1|_{TM_1}=b_{\bar k}g_1,~~T_2|_{TM_1}=a_{\bar k}g_1, ~ \text{ and } T_2|_{TM_2}=-b_{\bar k} T_1|_{TM_2}-a_{\bar k}g_2.
$$
\end{Theorem}

\begin{proof} We claim that the set $S_{(\bar k)}(\varepsilon)$ has at most two pairs. One only needs to prove this claim for $\varepsilon
\neq \varepsilon_1$. Assume otherwise $S_{(\bar k)}(\varepsilon)$ has at least three pairs, say,
$$
S_{(\bar k)}(\varepsilon)=\{(a_{\bar k},b_{\bar k}),(a_{i_1},b_{i_1}),\cdots,(a_{i_h},b_{i_h})\}
$$
and $b_{i_1}<b_{i_2}<\cdots<b_{i_h}$, for some $\varepsilon\neq\varepsilon_1$ and $h\geq 2$.

Let $d_j= a_j - \varepsilon_1b_j$ for $j =1, 2, \cdots, n$. Then it is clear that
$d_j = d_1$ when $j\in[\bar 1]\cup[\bar k]$. On the other hand, if $j\notin [\bar 1]\cup[\bar k]$, then
$$
d_1-d_j=(b_1-b_j)(-\varepsilon_1+\frac{a_1-a_j}{b_1-b_j}) <0.
$$
Because $\varepsilon_1$ is the smallest slope among all lines passing through $(a_1, b_1)$. Therefore
\begin{equation}\label{db1}
d_1=min\{d_1,d_2,\cdots,d_n\}.
\end{equation}
Now,  for any $i\in[\bar k]$ and $j\in[i_1]\cup[i_2]\cdots\cup[i_h]$, we have
\begin{equation}\label{dd3}
0 < d_j-d_i=(b_j-b_{\bar k})(\varepsilon-\varepsilon_1).
\end{equation}
Let us assume $\varepsilon > \varepsilon_1$ first, which immediately implies $b_{\bar k}<b_{i_1}<\cdots<b_{i_h}$ and hence $b_{\bar k} + \varepsilon <0$
from Lemma \ref{le5}. We then calculate
$$
0 >b_{\bar k} +\varepsilon > b_{\bar k} +\varepsilon_1 = b_{\bar k} + \frac{a_{\bar k} - a_{\bar 1}}{b_{\bar k} - b_{\bar 1}} = \frac {b_{\bar k}^2 +2a_{\bar k}}{b_{\bar k}-b_{\bar 1}},
$$
where we used
\begin{equation*}
b_{\bar 1}b_{\bar k} + a_{\bar 1} + a_{\bar k} =0
\end{equation*}
from Lemma \ref{le2}. Therefore $b_{\bar k}^2 +2a_{\bar k} <0$.
Meanwhile, applying Lemma \ref{le31} to $S_{(\bar k)}(\varepsilon)$, we have,
$$
b_{\bar k}b_{i_1}+a_{\bar k}+a_{i_1}\geq 0.
$$
Then
$$
0<b_{\bar k}b_{i_1}+a_{\bar k}+a_{i_1}-(b_{\bar k}^2+2a_{\bar k})=(b_{i_1}-b_{\bar k})(b_{\bar k}+\varepsilon),$$
which implies that $b_{\bar k}+\varepsilon> 0$. This is a contradiction and concludes that $\varepsilon < \varepsilon_1$.

From (\ref{dd3}), when $\varepsilon < \varepsilon_1$, we have $b_{i_1}<b_{i_2}<\cdots<b_{i_h}<b_{\bar k}$ and hence $b_{\bar k}+ \varepsilon > 0$ from
Lemma \ref{le5}. Similarly, applying Lemma \ref{le31} to $S_{(\bar k)}(\varepsilon)$ again, we have
$$
b_{\bar k}b_{i_1}+ a_{\bar k}+a_{i_1} \leq 0.
$$
Then
$$
0 \geq b_{\bar k}b_{i_1}+ a_{\bar k}+a_{i_1}  -(b_{\bar k}b_{\bar 1}+a_{\bar k}+a_{\bar 1})=(b_{i_1}-b_{\bar 1})(b_{\bar k}+ \frac{a_{i_1} - a_1}{b_{i_1} - b_1}),
$$
which implies that $b_{\bar k}+\varepsilon_1\leq0$ and therefore $b_{\bar k}+\varepsilon < b_{\bar k}+\varepsilon_1 \leq0$.
This is a contradiction again and concludes that no $S_{(\bar k)}(\varepsilon)$ has more than two pairs.

Consequently, as in the proof of Theorem \ref{case1}, we know $S_{(\bar 1)}(-b_{\bar k})$ contains all pairs except $(a_{\bar k}, b_{\bar k})$ and $
a_j = - b_{\bar k}b_j - a_{\bar k}$ for $j\notin[\bar k]$. By the minimality of $\varepsilon_1$, we find that $b_{\bar k}+\varepsilon_1 < 0$, which implies
$$
0 > b_{\bar k}+\varepsilon_1=b_{\bar k}+\frac{a_{\bar k}-a_1}{b_{\bar k }-b_1}=\frac{b_{\bar k}^2+2a_{\bar k}}{b_{\bar k}-b_1}
$$
and therefore
$$
b_{\bar k}^2+2a_{\bar k}<0.
$$

Finally, by a bootstrapping argument similar to that in the proof of Theorem \ref{case1}, repeatedly using the Cartan identity \eqref{tensor2}
and \eqref{abco1} in Lemma \ref{le2}, we can show that
$$
T_{ij,m}=0,~~when~~i\in[\bar k],~1\leq j,m\leq n,
$$
which implies that
\begin{equation}\label{cura1}
\omega_{ij}=0,~~when~~i\in [\bar k], ~j\notin [\bar k].
\end{equation}
Thus both the distribution $V_{b_{\bar k}}$ and its orthogonal complement are integrable and parallel according to \cite{sn}. The proof is completed
\end{proof}


\section{Laguerre invariants of hypersurfaces in $\mathbb{R}^{n+1}$}\label{lag-inv}

In this section we first recall Laguerre invariants of hypersurfaces in $\mathbb{R}^{n+1}$. For the details we refer readers to
\cite{lit1,lit2}. We then present the characterization of Dupin hypersurfaces with constant Laguerre curvatures in terms of
Laguerre invariants.

The group of Laguerre transformations is not as well known as the group of M\"{o}bius transformations. It is the other important subgroup of
the group of Lie sphere transformations. Let us first introduce the group of Lie sphere transformations (cf. \cite{cecil3}).  It starts with the space of all oriented
hyperspheres in $\mathbb{R}^{n+1}$, which are, points, oriented n-spheres, and oriented hyperplanes in $\mathbb{R}^{n+1}$. One may use the so-called
Lie quadric $Q_{n+1}$ to represent the space of all oriented hyperspheres. The Lie quadric is the projectivized light cone $C^{n+3}$ in the Minkowski spacetime
$\mathbb{R}^{n+4}_2$, where the Minkowski spacetime $\mathbb{R}^{n+4}_2$ is the vector space $R^{n+4}$
equipped with the quadratic
$$
<x, y> = - x_1y_1 + x_2y_2 + \cdots + x_{n+3}y_{n+3} - x_{n+4} y_{n+4}
$$
and the light cone is given as
$$
C^{n+3} = \{x\in \mathbb{R}^{n+4}| <x, x> = 0\}.
$$
The group of Lie sphere transformation is the orthogonal group $O(n+2, 2)/\{\pm1\}$ of the Minkowski spacetime $\mathbb{R}^{n+4}_2$. And the group
of Laguerre transformations is the isotropy subgroup of $O(n+2, 2)/\{\pm 1\}$ at $\mathfrak{p} = (1, -1, \vec{0}, 0)\in C^{n+3}\subset\mathbb{R}^{n+4}_2$.

A more geometric way to introduce the group of Lie sphere transformations is to consider the unit tangent bundle $U\mathbb{R}^{n+1}$ over $\mathbb{R}^{n+1}$, which
represents the space of lines on the Lie quadric $Q_{n+1}$.
It is clear that
$$
U\mathbb{R}^{n+1} = \mathbb{R}^{n+1}\times \mathbb{S}^n = \{(x,\xi)| x\in \mathbb{R}^{n+1}, \xi\in \mathbb{S}^n\}\subset \mathbb{C}^{n+1}$$
and there is a standard contact structure on $U\mathbb{R}^{n+1}$ defined by the standard contact form
$$
\omega=dx\cdot\xi.
$$
We then recall that oriented hypershperes in $U\mathbb{R}^{n+1}$ are the following three types:
\begin{itemize}
\item oriented n-sphere $S(p,r)=\{(x,\xi)\in U\mathbb{R}^{n+1}| x-p=r\xi\}$ for a point $p\in \mathbb{R}^{n+1}$ and a nonzero real number $r$
\item point sphere $S(p, 0) = \{(p, \xi)\ U\mathbb{R}^{n+1}| \xi\in UT_p\mathbb{R}^{n+1}\} $ for a point $p\in \mathbb{R}^{n+1}$, a "sphere" of radius $0$,
\item oriented hyperplane $P(\xi, \lambda) = \{ (x, \xi)\in U\mathbb{R}^{n+1}| x\cdot \xi = \lambda\}$ for a fixed unit vector $\xi$ and a real number $\lambda$, a "sphere"
of infinite radius.
\end{itemize}
It turns out that the group of Lie sphere transformations is also the group of diffeomorphisms of $U\mathbb{R}^{n+1}$ that take oriented hyperspheres to
oriented hyperspheres and preserve the contact structure $\omega$. Particularly, a Laguerre transformation is a Lie sphere transformation
that takes oriented spheres to oriented spheres and takes oriented hyperplanes to oriented hyperplanes.

Let $x: M^n\to\mathbb{R}^{n+1}$ be an oriented hypersurface in $\mathbb{R}^{n+1}$ with
non-vanishing principal curvatures. Then the unit normal $\xi: M^n \to \mathbb{S}^n$
is an immersion and $x$ induces a Laguerre surface
$f=(x,\xi): M^n\to U\mathbb{R}^{n+1}$.
Let $x$ and ${\tilde x}$ be two oriented hypersurfaces in $\mathbb{R}^{n+1}$ with non-vanishing principal curvatures. We say $x$ and $\tilde x$ are
Laguerre equivalent, if there is a Laguerre transformation $\phi: U\mathbb{R}^{n+1}\to U\mathbb{R}^{n+1}$
such that $(\tilde x, \tilde \xi)=\phi\circ (x, \xi)$.

Let $x: M^n\to \mathbb{R}^{n+1}$ be an umbilical free hypersurface with non-vanishing
principal curvatures. Let $\{e_1,e_2,\cdots, e_n\}$ be the
orthonormal basis for $TM^n$ with respect to $dx\cdot dx$, consisting
of unit principal vectors. We write the structure equation of $x:
M^n\to \mathbb{R}^{n+1}$ by
$$ e_j(e_i(x))=\sum_{k}\Gamma^k_{ij}e_k(x)+\lambda_i\delta_{ij}\xi;~~
e_i(\xi)=-\lambda_ie_i(x),~~~ 1\le i,j,k\le n,$$ where
$\lambda_i\neq 0$ is the principal curvature corresponding to $e_i$. Let
$$ R_i=\frac{1}{\lambda_i} ~\text{ and }~
R=\frac{R_1+R_2+\cdots+R_{n}}{n}$$ be the curvature radius and
mean curvature radius of $x$.
As in \cite{lit1,lit2}, we call
\begin{equation}\label{laguerre-position}
Y=\rho(x\cdot\xi, -x\cdot\xi,\xi,1): M^n \to C^{n+3}\subset\mathbb{R}^{n+4}_2
\end{equation}
the Laguerre position vector of the hypersurface $x$, where $\rho=\sqrt{\sum_{i=1}^n(R_i-R)^2}$. It is important to realize
the following covariant property.

\begin{Theorem}\label{lag1} (\cite{lit2})
Let $x$ and $\tilde{x}$
 be two umbilical free oriented hypersurfaces in $\mathbb{R}^{n+1}$ with non-vanishing
principal curvatures. Then $x$ and $\tilde{x}$ are Laguerre
equivalent if and only if their Laguerre positions $Y$ and $\tilde Y$ are the same up to a Laguerre transformation.
\end{Theorem}

Let $Y$ the Laguerre position vector of a hypersurface $x: M\to
\mathbb{R}^{n+1}$. We want to build a natural orthogonal moving frame along the surface $Y$ in $\mathbb{R}^{n+4}_2$.
Analogous to the cases of M\"{o}bius geometry,
$$
g=<dY,dY>
$$
is then called the Laguerre metric and the null normal vector
$$ N=\frac{1}{n}\Delta Y+\frac{1}{2n^2} <\Delta Y,\Delta
Y>Y
$$
to the surface $Y$ in $\mathbb{R}^{n+4}_2$ that is paired with the tautological null normal $Y$ such that
$$<Y, N> = -1,$$
where  the Laplacian operator $\Delta$ is that of the Laguerre metric $g$.
In contrast to the cases of M\"{o}bius geometry, we have a constant null normal vector
$\mathfrak{p}$ and a canonical null normal vector
$$
\eta=(\frac{1}{2}(1+|x|^2),\frac{1}{2}(1-|x|^2), x, 0)+R\,(x\cdot\xi, -x\cdot\xi,\xi,1)
$$
such that
$$
<\eta, \mathfrak{p}> = -1 \text{ and }  <\eta, Y>=<\eta, N>=<\mathfrak{p}, Y>=<\mathfrak{p}, N>=0.
$$
Therefore, if let $\{E_1,E_2,\cdots,E_n\}$ be an orthonormal basis for $g=<dY,dY>$ that are tangent to $Y$
with dual basis $\{\omega_1,\omega_2,\cdots,\omega_n\}$, then, as given in \cite{lit1,lit2}, we
have the following orthogonal moving frame along $Y$ in $\mathbb{R}^{n+4}_2$
$$\
\{Y, N, E_1, E_2,\cdots,E_n, \eta,\mathfrak{p}\}.
$$
We next recall from \cite{lit1,lit2} the following structure equations:
\begin{eqnarray}
&&E_i(N)=\sum_jL_{ij}E_j+C_i\mathfrak{p};\label{2.2}\\
&&E_j(E_i)=L_{ij}Y+\delta_{ij}N
+\sum_k\Gamma^k_{ij}E_k+B_{ij}\mathfrak{p};\label{2.3}\\
&&E_i(\eta)=-C_iY+\sum_jB_{ij}E_j.\label{2.4}
\end{eqnarray}
Analogous to the cases of M\"{o}bius geometry, besides the Laguerre metric $g=<dY,dY>$, we have the following Laguerre
invariants:
\begin{equation}\label{laguerre-inv}
\mathbb{B}=\sum_{ij}B_{ij}\omega_i\otimes\omega_j, ~~
\mathbb{L}=\sum_{ij}L_{ij}\omega_i\otimes\omega_j, \text{ and }  \mathbb{C}=\sum_iC_i\omega_i,
\end{equation}
where $\mathbb{B}$ is called the Laguerre second fundamental form, $\mathbb{L}$ is called the Laguerre tensor,  and $\mathbb{C}$ is called the Laguerre form.
In \cite{lit2}, the integrability conditions for $\{\mathbb{L}, \mathbb{B}, \mathbb{C}\}$ are identified as
\begin{eqnarray}
&&L_{ij,k}=L_{ik,j};\label{2.5}\\
&&C_{i,j}-C_{j,i}=\sum_k (B_{ik}L_{kj}-B_{jk}L_{ki});\label{2.6}\\
&&B_{ij,k}-B_{ik,j}=C_j\delta_{ik}-C_k\delta_{ij};\label{2.7}\\
&&R_{ijkl}=L_{jk}\delta_{il}+L_{il}\delta_{jk}-L_{ik}\delta_{jl}-L_{jl}\delta_{ik};\label{2.8}\\
&&\sum_{ij}B_{ij}^2=1, \sum_iB_{ii}=0,\sum_iB_{ij,i}=(n-1)C_j;\label{2.9}.
\end{eqnarray}
where
$R_{ijkl}$ is the curvature tensor of $g$. More importantly, in \cite{lit2}, it was shown that, up to a Laguerre transformation,
an umbilical free oriented hypersurfaces in $\mathbb{R}^{n+1}$ with non-vanishing principal curvatures is completely determined
by the Laguerre invariants $\{g, \mathbb{B}\}$ when $n>2$ and by the Laguerre invariants $\{g, \mathbb{B}, \mathbb{L}\}$
when $n=2$.

Finally we recall from \cite{lit1, lit2} how $\{g, \mathbb{B},\mathbb{C}\}$ can be
calculated in terms of the geometry of $x$ in $\mathbb{R}^{n+1}$:
\begin{equation}\label{lac}
\begin{split}
&g=\sqrt {\sum_i(R_i-R)^2}III, ~~~~~ B_{ij}=\rho^{-1}(R_i-R)\delta_{ij}, \\
&C_i=-\rho^{-2}\{\tilde{E_i}(R)+\tilde{E_i}(log\rho)(R_i-R)\},
\end{split}
\end{equation}
here $\tilde{E}_i=R_ie_i$ and $\{e_i\}$ is an orthonormal frame with respect to the metric $dx\cdot dx$, consisting of unit principal vectors.

\begin{Definition}
Let $x: M^n\to \mathbb{R}^{n+1}$ be an immersed hypersurface with non-vanishing
principal curvatures, and the principal curvature radius
$\{R_1=\frac{1}{\lambda_1},R_2=\frac{1}{\lambda_2},\cdots,R_r=\frac{1}{\lambda_r}\}.$
For any three principal curvature radius $\lambda_i,\lambda_j,\lambda_s$, We define the Laguerre curvature of $x$
$$\Upsilon_{ijs}=\frac{R_i-R_j}{R_i-R_s}.$$
\end{Definition}

The eigenvalues of $\mathbb{B}$ are called the Laguerre principal curvatures of $x$. Then, from (\ref{lac}), the Laguerre principal
curvature $b_i=\rho^{-1}(R_i-R)$.
Hence
\begin{equation}\label{lag-cur-inv}
\Upsilon_{ijs}=\frac{R_i-R_j}{R_i-R_s}=\frac{b_i-b_j}{b_i-b_s},
\end{equation}
which implies that the Laguerre curvatures $\Upsilon_{ijs}$ are Laguerre invariants.

We now are ready to give the characterization of Dupin hypersurfaces with constant Laguerre curvatures in terms of Laguerre invariants.

\begin{PROPOSITION}\label{pro1}
Let $x: M^n\to \mathbb{R}^{n+1}$ be an oriented hypersurface with $r(\geq 3)$ distinct non-vanishing
principal curvatures. Then  $x$ is a Dupin hypersurface with constant Laguerre curvatures if and only if  its Laguerre form vanishes and all
Laguerre principal curvatures are constant.
\end{PROPOSITION}
\begin{proof} First of all it is easily seen that, with a proof that is almost identical to the proof of
Proposition \ref{pro4}, the Laguerre curvatures $\Upsilon_{ijs}$ are constant if and only if
the Laguerre principal curvatures are constant.

Secondly, from (\ref{lac}), we have
\begin{equation}\label{lac1}
\begin{split}
C_i&=-\rho^{-2}\{\tilde{E}_i(r)+\tilde{E}_i(\rho)\rho^{-1}(r_i-r)\}=-\rho^{-2}\{\tilde{E}_i(r)+\tilde{E}_i(\rho)b_i\}\\
&=-\rho^{-2}\{\tilde{E}_i(r)+\tilde{E}_i(\rho b_i)-\rho^{-1}\tilde{E}_i(b_i)\}=-\rho^{-2}\{\tilde{E}_i(r_i)-\rho^{-1}\tilde{E}_i(b_i)\}\\
&=-\rho^{-2}\{-\frac{\tilde{E}_i(\lambda_i)}{\lambda_i^2}-\rho^{-1}\tilde{E}_i(b_i)\}.
\end{split}
\end{equation}
Therefore the proof of Proposition \ref{pro1} can easily be completed.
\end{proof}

In \cite{song}, an immersed hypersurface is said to be a Laguerre isoparametric hypersurface if its Laguerre
form  vanishes and its Laguerre Principal curvatures  are all constant. And Song \cite{song} has classified the
Laguerre isoparametric hypersurfaces with two distinct principal curvatures.
In these terminology, Theorem \ref{pro1} then says that, a
hypersurface in $\mathbb{R}^{n+1}$ is a Dupin surface with constant Laguerre curvatures if and only if
it is a Laguerre isoparametric hypersurface. We would like to mention that Proposition \ref{pro1} recently has also been observed in \cite{mtenen}.


\section{Dupin Hypersurfaces with constant Laguerre curvatures}\label{proof-th2}

In this section we will present the proof of Theorem \ref{th2}. We will first study examples of Dupin hypersyrfaces with constant Laguerre curvatures. Then we will proceed similar to
the M\"{o}bius cases to study the isoparametric tensors. This time the proof will be significantly simpler than that in M\"{o}bius cases because of \eqref{2.8} in contrast to
\eqref{equa4}.

\subsection{Examples}\label{lag-example}
\noindent
{\bf Cyclide of Dupin}. For any integer $k$ with $1\leq k\leq n-1$, let
$$
\mathbb{H}^{n-k} = \{(v, w)\in \mathbb{R}^{n-k+1}| |v|^2 -w^2 = -1 \text{ and } w>0\}
$$
be the hyperboloid in the Minkowski space
$\mathbb{R}_1^{n-k+1}$. We then consider the hypersurface
\begin{equation}\label{cyclide-dupin}
x(u,v,w)=(\frac{u}{w}(1+w),\frac{v}{w}): \mathbb{S}^k\times \mathbb{H}^{n-k}\to \mathbb{R}^{n+1},
\end{equation}
where $u: \mathbb{S}^k\to \mathbb{R}^{k+1}$ is the standard embedding of round sphere.

As was verified in \cite{lit1}, the hypersurface given in \eqref{cyclide-dupin} is a Laguerre isoparametric hypersurface with two distinct Laguerre principal curvatures and,
in fact, is Lie equivalent to the classical cyclide of Dupin of characteristic $(k,n-k)$.

\vskip 0.1in
\noindent
{\bf Flat Laguerre isoparametric hypersurface}. For any positive integers $m_1,\cdots,m_s$ with $m_1+\cdots+m_s=n$ and any non-zero constants $\kappa_1,\cdots,\kappa_s$,
we consider the hypersurface
\begin{equation}\label{flat-laguerre}
x(u_1,u_2,\cdots,u_s)=\big(\varphi,((1+\varphi\kappa_1)u_1,\cdots,(1+\varphi\kappa_s)u_s)\big): \mathbb{R}^n\to \mathbb{R}^{n+1},
\end{equation}
where $$\varphi=\frac{\kappa_1|u_1|^2+\cdots+\kappa_s|u_s|^2}{\kappa_1^2|u_1|^2+\cdots+\kappa_s^2|u_s|^2+1}$$
for $(u_1,u_2,\cdots,u_s)\in \mathbb{R}^{m_1}\times \mathbb{R}^{m_2}\times \cdots \mathbb{R}^{m_s}=\mathbb{R}^n$.

Again, as was shown in \cite{lit1}, the hypersurface given in \eqref{flat-laguerre} is a Laguerre isoparametric hypersurface with $s$
distinct Laguerre principal curvatures and, moreover, its Laguerre metric $g = <dY, dY>$ is flat, that is, $R_{ijij}[g]=0$.
Hence we call such hypersurfaces flat Laguerre isoparametric hypersurfaces.

\subsection{Proof of Theorem \ref{th2}}

The proof of Theorem \ref{th2} similar to that of M\"{o}bius cases uses isoparametric tensors. From the integrability conditions for the Laguerre invariants
$\mathbb{L}, \mathbb{B}, \mathbb{C}$ and the Laguerre metric $g$ it is clear that the Laguerre second fundamental form $\mathbb{B}$ is an isoparametric tensor
on a Dupin hypersurface with constant Laguerre curvatures in the light of Proposition \ref{pro1}. For the definition and basic properties for isoparametric tensors readers
are referred to Appendix A. First we want to show that the Laguerre symmetric 2-tensor $\mathbb{L}$ is also an isoparametric tensor on a Dupin hypersurface with constant
Laguerre curvatures. From \eqref{2.5}, it suffices to show that the eigenvalues of $\mathbb{L}$ are all constant, according to Definition \ref{iso-t}.

\begin{Lemma}\label{pro2}
Let $x: M^n\to \mathbb{R}^{n+1}$ be a Laguerre isoparametric hypersurface. Then the eigenvalues of
the Laguerre tensor $\mathbb{L}$ are all constant.
\end{Lemma}
\begin{proof} This lemma in the cases when $x$ has more than 2 distinct principal curvatures follows from Remark \ref{codazzi-iso} after the proof
of Theorem \ref{blas}. Next we consider $r=2$. And
$$
(L_{ij}) = diag(\tau_1, \tau_2, \cdots, \tau_n) \text{ and } (B_{ij})=diag(\underbrace{b_1,\cdots,b_1}_{s},\underbrace{b_2,\cdots,b_2}_{n-s}).$$
Since $\mathbb{B}$ is an isoparametric tensor with two distinct eigenvalues, from Proposition \ref{isop1} and Proposition \ref{isop},
we know that
$$
R_{ijij}=0,~~1\leq i\leq s, ~~s+1\leq j\leq n.
$$
On the other hand, from (\ref{2.8}) we know $R_{ijij} = - \tau_i - \tau_j$ and then
$$
(L_{ij}) = diag(\tau_1,\tau_2,\cdots,\tau_n)=diag(\underbrace{\tau,\cdots,\tau}_k,\underbrace{-\tau,\cdots,-\tau}_{n-k}).
$$
Therefore
$$
E_k(\tau)= E_k(\tau_i) = L_{ii,k}= L_{ik,i} = \frac {2\tau}{b_1-b_2}B_{ik,i} = 0
$$
for $i\in\{1, 2, \cdots, s\}$ if $k\in\{s+1, s+2, \cdots, n\}$ and
$$
E_k(\tau)= -E_k(\tau_i) = - L_{ii,k}= - L_{ik,i} = \frac{2\tau}{b_2-b_1} B_{ik,i} = 0
$$
for $i\in \{s+1, s+2, \cdots, n\}$ and $i\neq k$ if $k\in\{1, 2, \cdots, s\}$, in the light of \eqref{tensor1}. Thus $\tau$ is constant.
\end{proof}

Now we know, under the assumptions of Lemma \ref{pro2}, $\mathbb{B}$ and $\mathbb{L}$ are commuting isoparametric tensors. We may choose a local
orthonormal frame so that
$$
\left(B_{ij}\right)=diag(b_1,\cdots,b_n)$$
$$\left(L_{ij}\right)=diag(\tau_1,\cdots,\tau_n)=diag(\tau_{\tilde{1}},\cdots,\tau_{\tilde{1}},\tau_{\tilde{2}},
\cdots,\tau_{\tilde{2}},\cdots,b_{\tilde{t}},\cdots,\tau_{\tilde{t}}),$$
where $t$ denotes the number of the distinct eigenvalues of $\mathbb{L}$. We then define the index set
$$\{i\}=\{k\in\{1,2,\cdots, n\}|\tau_k=\tau_i\}
$$
according to the repeated eigenvalues of Laguerre tensor $\mathbb{L}$ in this section. The main observation that leads to the proof of Theorem \ref{th2}
is the following:

\begin{Theorem}\label{pro3}
Let $x: M^n\to \mathbb{R}^{n+1}$ be a Laguerre isoparametric hypersurface. Then the Laguerre second fundamental form $\mathbb{B}$
is parallel.
\end{Theorem}
\begin{proof} First, from (\ref{2.8}) and Cartan identity \eqref{tensor4}, we have
\begin{equation}\label{car}
\sum_{j\notin \{i\}}\frac{R_{ijij}}{\tau_j-\tau_i}=\sum_{j\notin \{i\}}\frac{-\tau_j-\tau_i}{\tau_j-\tau_i}=\sum_{j\notin \{i\}}\frac{\tau_i^2-\tau_j^2}{(\tau_j-\tau_i)^2}=0.
\end{equation}
Let $\tau_{i_0}^2=max\{\tau_1^2,\cdots,\tau_n^2\}$. And let $i=i_0$ in equation (\ref{car}), we have
$$\sum_{j\notin (i_0)}\frac{\tau_{i_0}^2-\tau_j^2}{(\tau_j-\tau_{i_0})^2}=0,$$
which implies that $\tau_{i_0}^2-\tau_j^2=0$ for $j\in\{1, 2, \cdots, n\}$ and therefore $t\leq 2$.

If $t=2$, then
$$
(L_{ij}) = diag(\tau_1,\tau_2,\cdots,\tau_n)=diag(\underbrace{\tau,\cdots,\tau}_{s},\underbrace{-\tau,\cdots,-\tau}_{n-s})
$$
for some constant $\tau \neq 0$. Moreover, from Proposition \ref{isop1} and Proposition \ref{isop}, we know $\mathbb{L}$ is parallel and
\begin{equation}\label{r0}
R_{ijij}=0,~~~~1\leq
i\leq s,~~s+1\leq j\leq n.
\end{equation}
Now we switch the order of $\{E_1,\cdots,E_s\}$ such that
\begin{equation*}
\left(B_{ij}\right)_{1\leq i,j\leq s} =diag(b_1,\cdots,b_s)=diag(b_{\hat{1}},\cdots,b_{\hat{1}},b_{\hat{2}},
\cdots,b_{\hat{2}},\cdots,b_{\hat{l}},\cdots,b_{\hat{l}})
\end{equation*}
and $b_{\hat{1}}<b_{\hat{2}}<\cdots<b_{\hat{l}}$. Assume that $l\geq 2$. From the Cartan identity \eqref{tensor4} and \eqref{r0}, we have
$$\sum_{1\leq j\leq n, b_j\neq b_{\hat{1}}}\frac{R_{1j1j}}{b_j-b_{\tilde{1}}}=\sum_{1\leq j\leq s, b_j\neq b_{\tilde{1}}}\frac{-2\tau_{\tilde{1}}}{b_j-b_{\tilde{1}}}=0,$$
which leads to the contradiction to $\tau\neq 0$. Therefore $l=1$ and, in fact, the Laguerre second fundamental form $\mathbb{B}$
has at most two distinct principal curvatures. Thus, from Proposition \ref{isop1}, $\mathbb{B}$ is parallel.

If otherwise $t=1$, then we can assume $L_{ij}=\tau\delta_{ij}$. Let $b_1$ be the smallest eigenvalues of $\mathbb{B}$. We then consider the Cartan identity
\eqref{tensor4}
$$
\sum_{1\leq j\leq n, b_j\neq b_1}\frac{R_{1j1j}}{b_j-b_1}=\sum_{1\leq j\leq n, b_j\neq b_1}\frac{-2\tau}{b_j-b_1}=0$$
and conclude that $\tau=0$, which implies $R_{ijkl}=0$ and the Laguerre second fundamental form is parallel using Theorem \ref{lecoda}.
So the proof is complete.
\end{proof}

Finally Theorem \ref{th2} follows from Theorem \ref{pro3} and the classification result in \cite{lit1}.

\begin{Theorem}\label{par} (\cite{lit1})
Let $x: M^n\to \mathbb{R}^{n+1}$ be an umbilical free hypersurface with non-vanishing principal curvatures. If its Laguerre second fundamental form is parallel,
then locally $x$ is Laguerre equivalent to one of the following hypersurfaces,\\
(1) the Cyclide of Dupin $x: \mathbb{S}^k\times \mathbb{H}^{n-k}\to \mathbb{R}^{n+1}$ given in \eqref{cyclide-dupin}\\
(2) the flat Laguerre isoparametric hypersurface $x: \mathbb{R}^n\to \mathbb{R}^{n+1}$ given in \eqref{flat-laguerre}.
\end{Theorem}

\appendix

\section{Isoparametric tensors}

\subsection{Definition and properties of isoparametric tensors}\label{subsec:iso-property-1}

A symmetric $(0,2)$ tensor field $T=\sum_{ij}T_{ij}\omega_i\otimes\omega_j$ on a Riemannian manifold  $(M^n,g)$ is said to be a Codazzi tensor if it satisfies the
Codazzi equation
\begin{equation}\label{codazzi}
\nabla_XT(Y,Z)=\nabla_YT(X,Z),
\end{equation}
for arbitrary vector fields $X,Y,Z$, where $\nabla$ denotes the Riemannian connection.
\begin{Definition}\label{iso-t}
A Codazzi tensor on a Riemannian manifold $(M^n, \ g)$ is said to be an isoparametric tensor if the  eigenvalues  are all constant.
\end{Definition}
Let $T=\sum_{ij}T_{ij}\omega_i\otimes\omega_j$ be an isoparametric tensor on a Riemannian manifold $(M^n,g)$. We can choose a local orthonormal
basis $\{E_1,\cdots,E_n\}$ such that
$$\left(T_{ij}\right)=diag(b_{\bar{1}},\cdots,b_{\bar{1}},b_{\bar{2}},\cdots,b_{\bar{2}},\cdots,b_{\bar{r}},\cdots,b_{\bar{r}}),
$$
where $b_{\bar{1}}<\cdots<b_{\bar{r}}$ are constants. Hence
\begin{equation}\label{codac}
(b_i-b_j)\omega_{ij}=\sum_kT_{ij,k}\omega_k,
\end{equation}
which gives
\begin{equation}\label{tensor1}
\begin{split}
&T_{ij,k}=0,~~when~~[i]=[j],~or~[j]=[k],~or~[i]=[k],\\
&\omega_{ij}=\sum_k\frac{T_{ij,k}}{b_i-b_j}\omega_k=\sum_{k\notin [i],[j]}\frac{T_{ij,k}}{b_i-b_j}\omega_k,~~when~~[i]\neq[j],
\end{split}
\end{equation}
where $[i]:=\{m\in \{1, 2, \cdots, n\}|b_m=b_i\}.$ Consequently we have
\begin{PROPOSITION}\label{isop1}
Let $(M^n,g)$ be a Riemannian manifold. If an isoparametric tensor $T$ on $(M^n,g)$ has only two distinct eigenvalues, then
$T$ is parallel.
\end{PROPOSITION}

It is well-known that a nontrivial parallel 2-tensor on a Riemannian manifold induces a splitting of Riemannian structure. Namely,

\begin{PROPOSITION}\label{isop} (\cite[Chap.4]{sn}).
Let $(M^n,g)$ be a Riemannian manifold. If $T$ is a parallel symmetric $(0,2)$ tensor field on $(M^n,g)$,
then, locally,
$$(M^n,g)=(M_1,g_1)\times(M_2,g_2)\times\cdots\times(M_r,g_r).$$
and there exist $r$ constants $\lambda_1,\cdots,\lambda_r$ such that
$$T=\lambda_1g_1\oplus\lambda_2g_2\oplus\cdots\oplus\lambda_rg_r.$$
\end{PROPOSITION}

Meanwhile, one may calculate from \eqref{tensor1}, for $[i]\neq [j]$,
$$T_{ij,ij}=\sum_{k\notin [i],[j]}\frac{2T^2_{ij,k}}{b_k-b_i},~~~~T_{ij,ji}=\sum_{k\notin [i],[j]}\frac{2T^2_{ij,k}}{b_k-b_j},
$$
and
\begin{equation}\label{tensor2}
R_{ijij}=\sum_{k\notin [i],[j]}\frac{2T^2_{ij,k}}{(b_k-b_i)(b_k-b_j)},
\end{equation}
using the Ricci identity. It is important that one immediately sees from \eqref{tensor2} the following useful fact in this paper.

\begin{lemma}\label{lecoda1}
Let $T=\sum_{ij}T_{ij}\omega_i\otimes\omega_j$ be an isoparametric tensor on the Riemannian manifold $(M^n,g)$. Under the orthonormal basis
$\{E_1,\cdots,E_n\}$, the coefficients of $T$ have the following form
$$\left(T_{ij}\right)=diag(b_{\bar{1}},\cdots,b_{\bar{1}},b_{\bar{2}},\cdots,b_{\bar{2}},\cdots,b_{\bar{r}},\cdots,b_{\bar{r}}),
~~b_{\bar{1}}<\cdots<b_{\bar{r}}.$$
Then
\begin{equation}\label{tensor3}
\begin{split}
&R_{ijij}\leq 0,~~when~~i\in [\bar{1}],~j\in [\bar{r}],\\
&R_{ijij}\geq 0,~~when~~i\in [\bar{k}],~ j\in [\overline{k+1}],~~for~~\bar{k}=\bar{1},\cdots,\overline{r-1},
\end{split}
\end{equation}
where $[\bar i] = \{m| b_m = b_{\bar i}\}$.
\end{lemma}

As a consequence, for instance, one can obtain the following strong geometric constraints for a Riemannian manifold to have an isoparametric tensor.

\begin{Theorem}\label{lecoda}
Let $T=\sum_{ij}T_{ij}\omega_i\otimes\omega_j$ be an isoparametric tensor on a Riemannian manifold $(M^n,g)$.
If $(M^n,g)$ has non-negative sectional curvature (or
non-positive sectional curvatures), then $T$ is parallel.
\end{Theorem}
\begin{proof}
Let us first present a proof in the cases when $(M^n,g)$ has non-positive sectional curvature.
We start with $i\in[\bar{1}]$ and $j\in[\bar{2}]$ in equation (\ref{tensor2}).  Notice that
$$(b_k-b_{\bar{1}})(b_k-b_{\bar{2}})>0,~~when~~k\notin [\bar{1}]\cup [\bar{2}].$$
From \eqref{tensor2} we therefore observe that
$$T_{ij,k}=0,~~when~~i\in[{\bar{1}}],~~j\in[{\bar{2}}],~1\leq k\leq n.$$
We then consider $i\in [\bar 1]$ and $j\in [\bar 3]$ in equation (\ref{tensor2}).
This time we notice that
$$(b_k-b_{\bar{1}})(b_k-b_{\bar{3}})>0,~~when~~k\notin [\bar{1}]\cup [\bar{2}] \cup[\bar 3]$$
and $T_{ij,k}=0,~i\in[{\bar{1}}],~k\in[{\bar{2}}]$. From \eqref{tensor2} again we observe that
$$T_{ij,k}=0,~~when~~i\in[{\bar{1}}],~~j\in[\bar 2]\cup[{\bar{3}}],~1\leq k\leq n.$$
Repeatedly we can prove that $T_{ij,k}=0$ for $i\in [\bar 1]$ and $j\in [\bar 2]\cup[\bar 3]\cdots \cup[\bar r]$. Similarly we
can prove $T_{ij,k} = 0$ for all indices, thus $T$ is parallel.

The proof for the cases when $(M^n, g)$ has non-negative curvature uses the same idea but one starts with $i\in [\bar 1]$
and $j\in [\bar r]$ instead.
\end{proof}

Next we want to derive the other important fact in this paper, the generalized Cartan identity (cf. \cite{li3}) on Riemannian manifolds with an isoparametric tensor  $T$.
\begin{equation}\label{tensor4}
\sum_{j\notin [i]}\frac{R_{ijij}}{b_j-b_i}=\sum_{j, k\notin [i]}\frac{2T^2_{ij,k}}{(b_k-b_i)(b_k-b_j)(b_i-b_j)}=0,
\end{equation}
which is easily seen as we note that the matrix, for each $i$ fixed,
$$\Big(\frac{2T^2_{ij,k}}{(b_k-b_i)(b_k-b_j)(b_i-b_j)}\Big)$$ is antisymmetric for indices $j,k$.

To show that the generalized Cartan identity is powerful in understanding the curvature structure of Riemannian manifolds with isoparametric tensors,
we present a proof of the following interesting result that is believed to be known.
Before we state the result we want to recall the Kulkarni-Nomizu product of tensors $T_1=\sum_{ij}T_{ij}\omega_i\otimes\omega_j$ and
$T_2=\sum_{ij}\hat{T}_{ij}\omega_i\otimes\omega_j$ defined by
$$(T_1\bigodot T_2)_{ijkl}=T_{ik}\hat{T}_{jl}+T_{jl}\hat{T}_{ik}-T_{il}\hat{T}_{jk}-T_{jk}\hat{T}_{il}.$$
It its then known that, on a locally conformally flat manifold $(M^n \ g)$ ($n\geq 3$),
\begin{equation}\label{con}
R_{ijkl}=(S\bigodot g)_{ijkl},
\end{equation}
where $S=\frac{1}{n-2}(Ric-\frac{R}{2(n-1)}g)$ is the so-called Schouten tensor.

\begin{Theorem}\label{cor3}
Let $(M^n,g),(n\geq 3)$ be a locally conformally flat Riemannian manifold. If the eigenvalues of the Schouten tensor $S$ are constant,
then either,\\
(a) $(M^n,g)$ is of constant curvature, or\\
(b) $(M^n,g)$ is locally reducible, $(M^n,g)=(M_1,g_1)\times(M_2,g_2)$, and there exists constant $\lambda$ such that
$$\pi_{1*}S=\lambda g_1,~~\pi_{2*}S=-\lambda g_2.$$
Where $\pi_1:M^n\to M_1,~~\pi_2:M^n\to M_2$ are the standard projections.
\end{Theorem}
\begin{proof} First of all, from the assumption that $g$ is locally conformally flat, we know that the Schouten tensor $S$ is an isoparametric tensor.
Hence, under a properly chosen a local orthonormal basis,
$$\left(S_{ij}\right)=diag(b_1,b_2,\cdots,b_n)=diag(b_{\bar{1}},\cdots,b_{\bar{1}},\cdots,b_{\bar{r}},\cdots,b_{\bar{r}}),
$$
for some constants $b_{\bar{1}}<\cdots<b_{\bar{r}}$. From (\ref{tensor4}) and \eqref{con}, for each $i$ fixed, we have
\begin{equation}\label{conc}
0 = \sum_{j\notin [i]}\frac{R_{ijij}}{b_j-b_i}=\sum_{j\notin [i]}\frac{b_j+b_i}{b_j-b_i}=\sum_{j\notin [i]}\frac{b_j^2-b_i^2}{(b_j-b_i)^2}.
\end{equation}
Now let $b_{i_0}^2=max\{b_1^2,\cdots,b_n^2\}$ and let $i=i_0$ in equation (\ref{conc}), we have
$$\sum_{j\notin [i_0]}\frac{b_j^2-b_{i_0}^2}{(b_j-b_{i_0})^2}=0,$$
which implies that $b_{i_0}^2-b_j^2=0$ and therefore $r\leq 2$. It is clear that $(M^n, \ g)$ is of constant curvature if $r=1$ and the proof is
complete in the light of Proposition \ref{isop1} and Proposition \ref{isop}.
\end{proof}


\subsection{Commuting isoparametric tensors}\label{subsec:iso-property-2}

Suppose that $(M^n, \ g)$ is a Riemannian manifold. Then we say that two $(0, 2)$-tensors are commuting if they are commuting as linear
transformations. Given two commuting isoparametric tensors
$$
T_1=\sum_{ij}T_{ij}\omega_i\otimes\omega_j \text{ and } T_2=\sum_{ij}\hat{T}_{ij}\omega_i\otimes\omega_j,
$$
we may choose a local orthonormal frame $\{E_1, E_2, \cdots, E_n\}$ so that
\begin{equation}\label{locbasis}
\left\{\begin{split}
&\left(T_{ij}\right)=diag(b_{\bar{1}},\cdots,b_{\bar{1}},b_{\bar{2}},\cdots,b_{\bar{2}},\cdots,b_{\bar{r}},\cdots,b_{\bar{r}})\\
&\left(\hat{T}_{ij}\right)=diag(a_1,\cdots,a_n)
\end{split}\right.
\end{equation}
for constants $b_{\bar 1} < b_{\bar 2} < \cdots < b_{\bar r}$ and $a_1, a_2, \cdots, a_n$. Immediately we know
\begin{equation}\label{cotensor}
\begin{split}
&T_{ij,k}=0,~~when~~[i]=[j],~or~[j]=[k],\\
&\hat{T}_{ij,k}=0,~~when~~a_i=a_j,~or~a_j=a_k,\\
&\frac{a_i-a_j}{b_i-b_j}T_{ij,k}=\hat{T}_{ij,k},~~when~~[i]\neq [j].
\end{split}
\end{equation}
Particularly the third equation in \eqref{cotensor} and $\hat T_{ij,k}=\hat T_{ik,j}$ implies
\begin{equation}\label{extra-vanish}
\hat T_{ij,k} = 0 \text{ for $[j] = [i]$ and $k\notin [j]$},
\end{equation}
which means more components of the commuting isoparametric tensors are forced to vanish and allows us to focus on the behavior of $T_2$ restricted to an eigenspace
of $T_1$
$$V_{b_i} = Span\{E_m: m\in [i]\} \text{ or } V_{b_{\bar k}} = Span\{E_m: m\in [\bar k]\}.
$$
We can change the order of the subbasis in the eigenspace $V_{b_{\bar{k}}}$ such that
$$\left(\hat{T}_{ij}\right)|_{i,j\in [\bar{k}]}=diag(\underbrace{a_{k_1},\cdots,a_{k_1}},\underbrace{a_{k_2},\cdots,a_{k_2}}\cdots,\underbrace{a_{k_m},\cdots,a_{k_m}})$$
for $a_{k_1}<a_{k_2}<\cdots<a_{k_m}.$
We then define the index sets
$$
(i):=\{l\in [i] | \quad a_l=a_i\} \text{ and }
(\bar{k_ i}) := \{l\in [\bar k]| \quad a_l = a_{k_i}\}.
$$
From \eqref{extra-vanish}, we have the following lemma.
\begin{lemma}\label{eigenspace} Suppose that $T_1$ and $T_2$ are two commuting isoparametric tensors as in the above. Then, for some $[\bar{k}]$ fixed,
$(i),(j)\in [\bar{k}]$ and $(i)\neq (j)$,
\begin{equation}\label{aa1}
(a_i-a_j)\omega_{ij}=\sum_{l\in[\bar{k}]}\hat{T}_{ij,l}\omega_l
\end{equation}
and
\begin{equation}\label{aa2}
R_{ijij}=\sum_{l\in[\bar{k}],l\notin (i),(j)}\frac{2\hat{T}_{ij,l}^2}{(a_i-a_l)(a_j-a_l)}.
\end{equation}
More importantly we have the generalized Cartan identity for $i\in[\bar{k}]$
\begin{equation}\label{aa4}
\sum_{j\in[\bar{k}],j\notin(i)}\frac{R_{ijij}}{a_i-a_j}=\sum_{j,l\in[\bar{k}],j,l\notin (i)}\frac{\hat{T}_{ij,l}^2}{(a_i-a_l)(a_j-a_l)(a_i-a_j)}=0.
\end{equation}
\end{lemma}

One important relation that ties two commuting isoparametric tensors more intimately to the geometry of the underlined manifold
and comes naturally from the integrability conditions \eqref{equa4} when we are concerned with the M\"{o}bius second fundamental form $T_1= B$
and the Blaschke tensor $T_2= A$ for a hypersurface in $f: M^n\to\mathbb{R}^{n+1}$ is
\begin{equation}\label{gauss}
R_{ijkl}=\frac{1}{2}(T_1\bigodot T_1)_{ijkl}+(T_2\bigodot g)_{ijkl}.
\end{equation}

\begin{lemma}\label{ble2} Suppose that $T_1$ and $T_2$ are two commuting isoparametric tensors on a Riemannian manifold
$(M^n, \ g)$ and satisfy the relation \eqref{gauss}. Then $T_2|_{V_{b_{\bar{k}}}}$ has two distinct eigenvalues at most. Moreover
$$b_{\bar{k}}^2+a_{\bar{k}}+\bar{a}_{\bar{k}}=0$$
when $T_2|_{V_{b_{\bar{k}}}}$ has two distinct eigenvalues $a_{\bar{k}}$ and $\bar{a}_{\bar{k}}$.
\end{lemma}

\begin{proof}
For $a_{k_1}<a_{k_2}<\cdots<a_{k_m}$ and $i\in (k_1)$ and $j\in (k_2)$, it is easily seen from \eqref{aa2} that
$$
R_{ijij}=\sum_{l\in[\bar{k}],l\notin (\bar{k_1}),(\bar{k_2})}\frac{2\hat{T}_{ij,l}^2}{(a_{k_1}-a_l)(a_{k_2}-a_l)}\geq 0.$$
Hence, from \eqref{gauss},
\begin{equation}\label{aa3}
R_{ijij}=b_{\bar{k}}^2+a_i+a_j\geq b_{\bar{k}}^2+a_{k_1}+a_{k_2} \geq 0, ~~i,j\in[\bar{k}] \text{ and } (i)\neq(j).
\end{equation}
Therefore, from the generalized Cartan identity \eqref{aa4} in Lemma \ref{eigenspace}, we get
\begin{equation}\label{aa5}
R_{ijij} = b^2_{\bar k} + a_{k_1} + a_j =0, ~~i\in (\bar{k_1}) \text{ and } j\in[\bar{k}],j\notin(\bar{k_1}).
\end{equation}
The key of this proof is to realize that \eqref{aa5} allows us to further trim the generalized Cartan identity \eqref{aa4} for $i\in (\bar{k_2})$ into
\begin{equation}\label{aa6}
\sum_{j\in[\bar k], j\notin(\bar{k_1}), j\notin(\bar{k_2}))}\frac{R_{ijij}}{a_{k_2}-a_{j}} =0,
\end{equation}
which in turn implies
$$
R_{ijij} = b^2_{\bar k} + a_{k_2} + a_j =0, ~~i\in (\bar{k_2}) \text{ and } j\in[\bar{k}],j\notin(\bar{k_2}).
$$
Thus, repeating the above argument,  we can get
\begin{equation}\label{aa7}
R_{ijij} = b^2_{\bar k} + a_i + a_j =0 \text{ for all } i,j\in[\bar{k}] \text{ and } (i)\neq(j),
\end{equation}
which forces $m\leq 2$ and completes the proof.
\end{proof}

{\bf Acknowledgements:} The first author and third author are partially supported by the grant No.11471021 and No. 11171004 of NSFC. And
the second author is partially supported by NSF grant DMS-1303543.


\begin{thebibliography}{11}
\bibitem{ak}Akivis M. A., Goldberg V.V., {\sl A conformal differential invariants and the conformal rigidity of hypersurfaces},
Proc. Amer. Math. Soc.,125,(1997),2415-2424.
\bibitem{car0}Cartan E., {\sl Familles de surfaces isoparam\'{e}triques dans les espaces \`{a} courbure constante}, Annali di Mat., 17,
(1938), 177-191 (see also in Oeuvres Compl\`{e}tes, Partie III, Vol. 2, 1431-1445).
\bibitem{cecil3}Cecil T., {\sl Lie sphere geometry, with applications to submanifolds,} 2nd ed., Universitext, Springer, New York,
2008.
\bibitem{cecil4}Cecil T., Chern S.S., {\sl Tautness and Lie sphere geometry,} Math. Ann., 278, (1987), 381-399.
\bibitem{cecil5}Cecil T., Chi Q.S., Jensen G., {\sl Isoparametric hypersurfaces with four principal curvatures,} Ann. of Math., (2)
166, (2007), 1-76.
\bibitem{cecil6}Cecil T., Chi Q.S., Jensen G., {\sl Dupin hypersurfaces with four principal curvatures. II,} Geom. Dedicata 128,
(2007), 55-95.
\bibitem{cecil6.5} Cecil T., Chi Q.S., Jensen G., {\sl On Kuiper's question whether taut submanifolds are algebraic,} Pacific J. Math.,
234(2), (2008), 229 - 248.
\bibitem{cecil7}Cecil T., Chi Q.S., Jensen G., {\sl Classifications of Dupin hypersurfaces,} in Pure and Applied Differential
Geometry, PADGE 2007, Editors F. Dillen and I. van de Woestyne, Shaker Verlag, Aachen, 2007, 48-56.
\bibitem{cecil8}Cecil T., Jensen G., {\sl Dupin hypersurfaces with three principal curvatures,} Invent. Math., 132, (1998), 121-178.
\bibitem{cecil9}Cecil T., Jensen G., {\sl Dupin hypersurfaces with four principal curvatures,} Geom. Dedicata 79, (2000), 1-49.
\bibitem{cecil10}Cecil T., Ryan P., {\sl Focal sets, taut embeddings and the cyclides of Dupin,} Math. Ann., 236(2), (1978), 177 - 190.
\bibitem{Dorf}Dorfmeister J., Neher E., {\sl An algebraic approach to isoparametric hypersurfaces. I, II,} T\^{o}hoku Math. J., (2) 35,
(1983), 187-224, 225-247.
\bibitem{Dupin}Dupin C., {\sl Applications de g\'{e}om\'{e}trie et de m\'{e}chanique,} Bachelier, Paris, 1822.
\bibitem{gro}Grove K., Halperin S., {\sl Dupin hypersurfaces, group actions, and the double mapping cylinder,} J. Differential
Geom., 26, (1987), 429-459.
\bibitem{hu1}Hu Z., Li D., {\sl M\"{o}bius isoparametric hypersurfaces with three distinct principal curvatures,} Pacific J. Math.,
232, (2007), 289-311.
\bibitem{hu2}Hu Z.J., Zhai S. J., {\sl Moebius isoparametric hypersurfaces
 with three distinct principal curvatures,II},
Pacific J. Math., 249,(2011), 343-370.
\bibitem{hu3}Hu Z., Li H.Z., Wang C.P., {\sl Classification of M\"{o}bius isoparametric hypersurfaces in $S^5$,} Monatsh. Math.,
151, (2007), 201-222.
\bibitem{im}Immervoll S., {\sl On the classification of isoparametric hypersurfaces with four distinct principal curvatures in
spheres,} Ann. Math., 168, (2008), 1011-1024.
\bibitem{sn}Kobayashi S. and
Nomizu,{\sl Foundations of Differential Geometry,(I),} Wiley
Interscience 1969.
\bibitem{lih2}Li H.Z., Liu H.L., Wang C.P.  and  Zhao G.S., {\sl M\"{o}bius isoparametric hypersurfaces
in $S^{n+1}$ with two distinct principal curvatures}, Acta Math. Sinica, English Series
18(2002), 437-446.
\bibitem{li2}Li H.Z., Wang C.P., {\sl M\"{o}bius geometry of hypersurfaces with constant mean curvature and scalar curvature}, Manuscripta Math.,
112,(2003),1-13.
\bibitem{li3}Li H.Z., {\sl Generalized Cartan identities on isoparametric manifolds}, Ann. Glob. Anal. Geom., 15(1997), 45-50.
\bibitem{lit1}Li T.Z., Li H.Z., Wang C.P., {\sl Classification of hypersurfaces with parallel Laguerre second fundamental form in $R^n$},
Differential Geom. Appl., 28,(2010), 148-157.
\bibitem{lit2}Li T.Z., Wang C.P., {\sl Laguerre geometry of hypersurfaces in $R^n$}, Manuscripta Math., 122, (2007), 73-95.
\bibitem{liu}Liu H.L., wang C.P., Zhao G.S., {\sl M\"{o}bius isotropic submanifolds in
$S^n$},Tohoku Math. J., 53,(2001),553-569.
\bibitem{miya1}Miyaoka R., {\sl Compact Dupin hypersurfaces with three principal curvatures,} Math. Z., 187, (1984), 433-452.
\bibitem{miya2}Miyaoka R., {\sl Dupin hypersurfaces and a Lie invariant,} Kodai Math. J., 12, (1989), 228-256.
\bibitem{miya3}Miyaoka R., {\sl Dupin hypersurfaces with six principal curvatures,} Kodai Math. J., 12, (1989), 308-315.
\bibitem{miya4}Miyaoka R., Ozawa T., {\sl Construction of taut embeddings and Cecil-Ryan conjecture,} In: Geometry of manifolds (Matsumoto, 1988),
Perspect. Math., vol. 8, pp. 181 - 189. Academic Press, Boston, MA (1989).
\bibitem{mu1}M\"{u}nzner H.F., {\sl Isoparametrische Hyperfl\"{a}chen in Sph\"{a}ren,} Math. Ann., 251,(1980), 57-71.
\bibitem{mu2}M\"{u}nzner H.F., {\sl Isoparametrische Hyperfl\"{a}chen in Sph\"{a}ren II: \"{U}ber die Zerlegung der Sph\"{a}re in Ballb\"{u}ndel,}
Math. Ann., 256, (1981), 215-232.
\bibitem{mus}Musso E., Nicolodi L., {\sl A variational problem for surfaces in Laguerre geometry,}
Trans. Amer. Math. Soc., 348, (1996), 4321-4337.
\bibitem{Nie1}Niebergall R., {\sl Dupin hypersurfaces in $R^5$. I, II,} Geom. Dedicata 40, (1991), 1-22, 41, (1992), 5-38.
\bibitem{pink1}Pinkall U., {\sl Dupin¡¯sche Hyperfl¡§achen in $E^4$,} Manuscripta Math., 51, (1985), 89-119.
\bibitem{pink2}Pinkall U., {\sl Dupin hypersurfaces,} Math. Ann., 270, (1985), 427-440.
\bibitem{pink3}Pinkall U., Thorbergsson G., {\sl Deformations of Dupin hypersurfaces,} Proc. Amer. Math. Soc., 10, (1989),
1037-1043.
\bibitem{mtenen}Cezana M.Jr., Tenenblat K., {\sl A characterization of Laguerre isoparametric hypersurfaces of the Euclidean space,} Monatsh. Math., 175, (2014),187-194.
\bibitem{riv3}Riveros C.M.C., Tenenblat K., {\sl Dupin hypersurfaces in $R^5$,} Canad. J. Math., 57, (2005), 1291-1313.
\bibitem{rod}Rodrigues L.A., Tenenblat K., {\sl A characterization of Moebius isoparametric hypersurfaces of the sphere,} Monatsh. Math., 158, (2009),321-327.
\bibitem{st}Stolz S., {\sl Multiplicities of Dupin hypersurfaces,} Invent. Math., 138, (1999), 253-279.
\bibitem{song}Song Y.P., {\sl Laguerre isoparametric hypersurfaces in $R^n$ with two distinct non-zero principal curvatures}, Acta Math. Sin. (Engl. ser),
30,(2014),169-180.
\bibitem{te1}Terng C.L., {\sl Isoparametric submanifolds and their Coxeter groups,} J. Differential Geom., 21, (1985), 79-107.
\bibitem{tho1}Thorbergsson G., {\sl Dupin hypersurfaces,} Bull. London Math. Soc., 15, (1983), 493-498.
\bibitem{tho2}Thorbergsson G., {\sl Isoparametric foliations and their buildings,} Ann. Math., 133, (1991), 429-446.
\bibitem{o}O'Neil B., {\sl Semi-Riemannian Geometry}, Academic Press, New York(1983).
\bibitem{w}Wang C.P., {\sl Moebius geometry of submanifolds in
$S^n$}, Manuscripta Math., 96,(1998), 517-534.
\end{thebibliography}
\end{document}